\documentclass[11pt,a4paper]{article}

\usepackage[T1]{fontenc}
\usepackage[utf8]{inputenc}
\usepackage{lmodern}
\usepackage[a4paper,margin=29mm]{geometry}
\usepackage{amsmath,amssymb,amsthm,mathtools}
\usepackage{microtype}
\usepackage{enumitem}
\usepackage{booktabs}
\usepackage{xurl}
\usepackage[hidelinks]{hyperref}

\hypersetup{
  pdftitle={Diophantine equations in matrix rings and semirings},
  pdfauthor={Mihai Prunescu},
  pdfsubject={Noncommutative Diophantine equations, central polynomials, and undecidability},
  pdfkeywords={Hilbert's tenth problem, matrix rings, matrix semirings, central polynomials, undecidability}
}

\newtheorem{theorem}{Theorem}[section]
\newtheorem{lemma}[theorem]{Lemma}
\newtheorem{proposition}[theorem]{Proposition}
\newtheorem{corollary}[theorem]{Corollary}
\theoremstyle{definition}
\newtheorem{definition}[theorem]{Definition}
\theoremstyle{remark}
\newtheorem{remark}[theorem]{Remark}

\newcommand{\N}{\mathbb N}
\newcommand{\Z}{\mathbb Z}
\newcommand{\Q}{\mathbb Q}

\newcommand{\Id}{I_n}
\newcommand{\free}[1]{\Z\langle #1\rangle}
\newcommand{\Eq}{\mathsf{Eq}}
\newcommand{\Sys}{\mathsf{Sys}}
\newcommand{\Hten}{\mathsf{H10}}
\newcommand{\FEq}{\mathsf{F}}
\newcommand{\FSys}{\mathsf{SF}}
\newcommand{\HEq}{\mathsf{H}}
\newcommand{\HSys}{\mathsf{SH}}
\DeclareMathOperator{\tr}{tr}

\DeclareMathOperator{\sgn}{sgn}
\numberwithin{equation}{section}
\setlist[enumerate]{label=\textup{(\roman*)},leftmargin=*,itemsep=3pt}
\title{Diophantine equations in matrix rings and semirings}
\author{Mihai Prunescu\thanks{Research Center for Logic, Optimization
and Security (LOS), Faculty of Mathematics and Computer Science,
University of Bucharest, Academiei 14, Bucharest (RO-010014), Romania.}
\thanks{Simion Stoilow Institute of Mathematics of the Romanian Academy,
Research unit 5, P. O. Box 1-764, Bucharest (RO-014700), Romania.
E-mail: \href{mailto:mihai.prunescu@imar.ro}{\nolinkurl{mihai.prunescu@imar.ro}}.}}
\date{}

\begin{document}
\maketitle

\begin{abstract}
For every fixed positive integer $n$, we prove that it is undecidable whether
a single noncommutative polynomial equation with scalar integer coefficients
has a solution in $M_n(\N)$ or in $M_n(\Z)$.  An integral range calculation
for central polynomials yields a common reduction from integer Diophantine
equations; the single-equation and finite-system problems are computably
enumerable complete.  The same polynomials have exact range $\Q I_n$ on
rational, even nonnegative rational, arguments.  Consequently, the
single-equation and finite-system problems over $M_n(\Q)$ are computably
many-one equivalent to Hilbert's tenth problem over $\Q$, unconditionally.
The rational undecidability claims reported in 2026 by OpenAI and
Zhi-Wei Sun imply rational-matrix undecidability, with at most $86$, $87$, and
$88$ matrix variables in dimensions $1$, $2$, and at least $3$, respectively.
An explicit Diophantine definition of $\Z$ in $\Q$ reported by the project
coordinated by Elliot Glazer further gives completeness among computably
enumerable problems.  These consequences of recent arithmetic claims
are stated conditionally.
For each $R\in\{\N,\Z,\Q\}$, inclusion of the solvability theories from
dimension $n$ to dimension $m$ holds exactly when $n$ divides $m$, both for
single equations and for systems.  We also treat additive restrictions on
entries and matrix subsemirings with scalar retracts.  For nontrivial
solutions, homogeneous equations and systems in dimension one over
$\N$, $\Z$, and $\Q$ are many-one equivalent to the rational scalar
problem; their matrix analogues stabilize and become decidable from
dimension two.  Constant-free equations form a strictly increasing chain
in dimension, and their single-equation problems over $\Z$ and $\Q$
are decidable from dimension two.  Constant-free systems are computably
enumerable complete over $\N$ and $\Z$, and many-one equivalent to Hilbert's
tenth problem over $\Q$ for rational matrices.
\end{abstract}



\section{Introduction}

The Davis--Putnam--Robinson--Matiyasevich theorem implies that there is no
algorithm deciding whether an arbitrary polynomial with integer coefficients
has an integer zero; see \cite{DPR1961,Matiyasevich1971,Matiyasevich1993}.
For matrices, one can ask the same question for noncommutative polynomials.
The answer depends on the domain of the entries, the permitted coefficients,
and whether the input consists of one equation or a finite system.  We work
with scalar integer coefficients: an integer $a$ is interpreted as $aI_n$,
and no noncentral matrix parameter is named.

Our first main result is single-equation undecidability over both
$M_n(\N)$ and $M_n(\Z)$, for every fixed $n\geq1$.  More precisely, we
effectively transform an ordinary polynomial $f$ into one noncommutative
polynomial $F_{n,f}$ such that
\begin{equation}\label{eq:intro-equivalence}
\begin{aligned}
 f=0\text{ has a solution in }\Z
 &\iff F_{n,f}=0\text{ has a solution in }M_n(\N)\\
 &\iff F_{n,f}=0\text{ has a solution in }M_n(\Z).
\end{aligned}
\end{equation}
The equivalence between the two matrix domains concerns this constructed
family of instances.  It is not an equivalence for arbitrary matrix equations.

Our second main result concerns rational matrices.  The central polynomials
used in the integral reduction have exact scalar range $\Q I_n$ on rational
matrix arguments, including arguments with nonnegative rational entries.
This gives, without any assumption on the decidability of scalar rational
equations,
\begin{equation}\label{eq:intro-rational-equivalence}
 \Eq_n(\Q)\equiv_m\Sys_n(\Q)\equiv_m\Hten(\Q),
\end{equation}
where $\Hten(\Q)$ is rational solvability of one integer-coefficient
polynomial, and $\equiv_m$ denotes computable many-one equivalence.
The rational counterpart of \eqref{eq:intro-equivalence} is a separate
statement with rational scalar solutions on its left-hand side.

Three recent arithmetic preprints give further applications.
OpenAI \cite[Theorem~1.1 and Corollary~7.1]{OpenAI2026Q} reports
undecidability of rational Diophantine equations, including equations of
degree at most four with an unrestricted number of variables.
Sun \cite[Theorem~1.1]{Sun2026Q} reports undecidability with at most
$86$ rational variables.  Shared auxiliary matrix variables transfer the
latter bound to $86$, $87$, and $88$ matrix variables for $n=1$, $n=2$,
and $n\geq3$, respectively.  Separately, the quartic restriction transfers
to degree bounds $4$, $16$, and $4(2n^3+n+1)$; no simultaneous bound on
degree and number of variables is asserted.
The AI-assisted project coordinated by Elliot Glazer
\cite[Theorem~1.1]{GlazerProject2026Q} reports an explicit polynomial
defining the ordinary integers as a Diophantine subset of $\Q$.
That stronger claim implies completeness among computably enumerable
problems under many-one reductions.  Section~\ref{sec:rational} states
the three claims precisely and derives their consequences conditionally.
The matrix reductions themselves are unconditional.  The arithmetic
proofs in the cited preprints have not been independently verified here.

Our third main result describes the dependence on dimension.  Let
$\Eq_n(R)$ and $\Sys_n(R)$ denote, respectively, the sets of single equations
and finite systems solvable over $M_n(R)$, with the same scalar-coefficient
syntax in every dimension.  For $R\in\{\N,\Z,\Q\}$ we prove
\begin{equation}\label{eq:intro-inclusions}
 \Eq_n(R)\subseteq\Eq_m(R)
 \quad\Longleftrightarrow\quad
 \Sys_n(R)\subseteq\Sys_m(R)
 \quad\Longleftrightarrow\quad n\mid m.
\end{equation}
Thus these families, ordered by inclusion, are isomorphic to the divisibility
lattice of the positive integers.  Incomparable dimensions give incomparable
theories, while every proper divisibility chain gives a strictly increasing
chain.  This describes literal inclusion of the sets of solvable instances;
it is independent of the computable equivalence of their decision problems.

The proof of \eqref{eq:intro-equivalence} uses the existence of
central polynomials for matrix algebras, established independently by
Formanek and Razmyslov \cite{Formanek1972,Razmyslov1973}.  We use Razmyslov's
multilinear form.  The values of such a polynomial on tuples of elementary matrices $E_{ij}$
determine a positive integer $d_n$.  B\'ezout's identity and copies on
pairwise disjoint variable sets yield a polynomial parametrization of $d_n\Z I_n$, with
every value realized using nonnegative matrix arguments.  Clearing a scalar
denominator gives \eqref{eq:intro-equivalence}.

Related model-theoretic results concern equational interpretations of matrix
rings.  Klep and Tressl \cite[Corollary~2.1.4]{KlepTressl2025} give a positive
primitive definition of the center of a matrix ring over a field, using
existentially quantified elementary matrices.  Lee \cite{Lee2024}, reporting joint
work with Miasnikov, announces mutual equational interpretations between a
unital ring and its full matrix ring.  Here the required scalar values are
obtained from polynomial expressions themselves, so that no equations
imposing scalarity have to be appended.  This is what permits a single
equation with scalar coefficients, also over the semiring $M_n(\N)$.

One cannot combine arbitrary matrix equations merely by adding squares of
their residuals: the nonzero matrix $E_{12}$ satisfies $E_{12}^2=0$ when
$n\geq2$.  All sum-of-squares arguments used for the undecidability reduction
below take place over scalar integers or rationals, or on expressions
already known to be scalar.  The same square-zero matrices have a different consequence for
homogeneous equations: their nontrivial-solvability problems stabilize in
dimension two.  In contrast, constant-free equations that need not be
homogeneous form a strict chain in the ordinary order of dimensions.
Nontrivial constant-free systems remain as hard as ordinary systems:
they are computably enumerable complete over $\N$ and $\Z$, and many-one
equivalent to $\Hten(\Q)$ over $\Q$.

The homogeneous problem in dimension one has a classical connection with
rational solvability.  Matiyasevich
\cite[Section~7.4, pp.~146--149]{Matiyasevich1993} proves that deciding
whether a homogeneous polynomial has an integer zero other than the zero
tuple is equivalent to deciding rational solvability of an arbitrary
polynomial.  We give a detailed proof, including the extra homogeneous
constraint needed to exclude a zero homogenizing coordinate, and an
explicit variant realizing nonnegative integer witnesses.  Consequently,
for $R\in\{\N,\Z,\Q\}$,
\begin{equation}\label{eq:intro-scalar-homogeneous}
 \HEq_1(R)\equiv_m\HSys_1(R)\equiv_m\Hten(\Q),
\end{equation}
where $\HEq_1(R)$ and $\HSys_1(R)$ denote nontrivial solvability of
homogeneous scalar equations and finite systems, respectively.
Thus the cited claim of rational undecidability implies undecidability
of these dimension-one problems, and the reported Diophantine definition
implies their completeness among computably enumerable problems.
The equivalences themselves are
unconditional.  Section~\ref{subsec:scalar-homogeneous} gives the complete
construction before treating the decidable higher-dimensional cases.

The paper is organized as follows.  Section~\ref{sec:definitions} fixes
the notation and defines the decision problems.
Section~\ref{sec:dimension} proves that inclusion of solvability theories
is characterized by divisibility of dimensions.
Section~\ref{sec:central} constructs central polynomials with exact
integral scalar ranges.  Section~\ref{sec:main} uses them to prove
single-equation undecidability over $M_n(\N)$ and $M_n(\Z)$.
Section~\ref{sec:dimension-two} gives explicit polynomials and numerical
examples in dimension two.  Section~\ref{sec:systems} reduces finite
systems to single equations and establishes computably enumerable
completeness.  Section~\ref{sec:rational} proves the rational scalar-range
and reduction theorems, then derives conditional undecidability and
completeness results from the recent arithmetic claims, with separate
bounds on degree and number of variables.
Section~\ref{sec:restricted} treats restrictions on entries through
additive bases and scalar retracts.  Section~\ref{sec:corners} defines
diagonal elementary matrices existentially in $M_n(\N)$ and gives an explicit
reduction to three equations.  Section~\ref{sec:nontrivial} determines
the dimension dependence and decidability properties of constant-free
and homogeneous equations and systems with nontrivial solutions,
including the scalar homogeneous equivalence with $\Hten(\Q)$ and the
homogeneous stabilization in dimension two.
Appendix~\ref{app:parameters} constructs explicit range polynomials,
tabulates $d_n$ and $s_n$ for $2\leq n\leq5$, and derives formulas and
growth estimates for these parameters and the polynomial degrees.

\section{Equations, coefficients, and decision problems}
\label{sec:definitions}

Throughout, $\N=\{0,1,2,\ldots\}$, $\Z$ is the ring of integers,
and $\Q$ is the field of rational numbers.  Matrix dimensions are positive
integers.  For $R\in\{\N,\Z,\Q\}$, $M_n(R)$ denotes the set of
$n\times n$ matrices with entries in $R$, with ordinary matrix addition
and multiplication.  It is a semiring for $R=\N$ and a ring for
$R=\Z,\Q$.  We write $A_{ij}$ for the entry in row $i$ and column $j$
of a matrix $A$, and $I_n$ and $0_n$ for the identity and zero matrices.
Only $0_n$ and $I_n$ are named matrix constants in the language; a bare
$0$ denotes the zero scalar or zero matrix according to context.

For every positive dimension $r$, let $e_i\in\Q^r$ ($1\leq i\leq r$)
be the standard \emph{column} vector with its $i$th entry equal to $1$
and all other entries equal to $0$.  For $1\leq i,j\leq r$, the
\emph{matrix unit} $E_{ij}\in M_r(\N)$ has a $1$ in position $(i,j)$
and zeros elsewhere.  More explicitly, if $\delta_{ab}$ is the Kronecker
delta, equal to $1$ when $a=b$ and to $0$ otherwise, then, for indices
between $1$ and $r$,
\begin{equation}\label{eq:matrix-unit-notation}
 (e_i)_a=\delta_{ai},\qquad
 (E_{ij})_{ab}=\delta_{ai}\delta_{bj},\qquad
 E_{ij}e_k=\delta_{jk}e_i,\qquad
 E_{ij}E_{k\ell}=\delta_{jk}E_{i\ell}.
\end{equation}
The ambient dimension of $e_i$ and $E_{ij}$ is always that of the
matrix space under discussion.  These vectors and matrices are used in
proofs and are not additional named constants.

For square matrices $A_1,\ldots,A_t$, the notation
$\operatorname{diag}(A_1,\ldots,A_t)$ denotes the block diagonal matrix
with these diagonal blocks and zeros off the blocks; scalar arguments
give an ordinary diagonal matrix.  An empty block $0_0$ is omitted.
Empty sums are zero, and empty products and zeroth powers are the
appropriate identity.  Products of noncommuting factors indexed by an
integer interval are taken in increasing index order.
For $A\in M_r(\Q)$, we write
$\tr(A)=\sum_{i=1}^r A_{ii}$ for the trace, $\det(A)$ for the determinant,
and $\operatorname{im}_{\Q}A=\{Av:v\in\Q^r\}$ for the image of a
rational $r\times r$ matrix; its rank is the dimension of this image.
The notation $\dim_K V$ denotes the dimension of a vector space $V$
over a field $K$.  We write $S_r$ for the group of permutations of
$\{1,\ldots,r\}$ and $\sgn(\sigma)\in\{1,-1\}$ for the sign of a
permutation $\sigma$.  Bold letters such as $\mathbf X$ denote finite
tuples of the indicated objects.

For dimension-dependent polynomial families such as $Q_n$, the subscript
records the matrix dimension.  Parenthesized superscripts label copies:
$Q_n^{(i)}$ is not a power.  A copy is obtained by renaming the variables,
and copies on \emph{pairwise disjoint variable sets} have no variable symbol
in common.  Their assigned matrix values need not be distinct.  Whenever
polynomial copies occur below, their variable blocks are specified.

\begin{definition}\label{def:equation}
The notation $\Z\langle X_1,\ldots,X_k\rangle$ denotes the free
associative unital ring of finite integer linear combinations of words
in the variables $X_1,\ldots,X_k$; multiplication of words is
concatenation.  A \emph{noncommutative polynomial with scalar integer
coefficients} is an element
\[
 P\in\free{X_1,\ldots,X_k}.
\]
The declared variable list is part of the input; its length is $k$,
which may be zero for the ordinary solvability problems.
Its monomials are these words, including the empty word.  Angle brackets
with coefficient set $\N$ or $\Q$ have the analogous meaning; square
brackets, as in $\Z[x_1,\ldots,x_k]$, denote ordinary commutative
polynomials.  The length of a word is its number of variable occurrences,
with the empty word having length zero.  The degree $\deg P$ for $P\ne0$
is the largest length of a word with nonzero
coefficient.  A polynomial is \emph{constant-free} when the coefficient
of the empty word is zero, and a nonzero polynomial is \emph{homogeneous
of degree $d$} when all its words with nonzero coefficients have length
$d$.  Under evaluation in $M_n(\Q)$, the empty word is $\Id$, and an
integer coefficient $a$ is interpreted as $a\Id$.

For $R\in\{\N,\Z,\Q\}$, the equation $P=0$ is \emph{solvable in $M_n(R)$}
if there are matrices $A_1,\ldots,A_k\in M_n(R)$ for which
$P(A_1,\ldots,A_k)=0$, with the polynomial evaluated in $M_n(\Q)$.
For integer or nonnegative integer arguments its value lies in $M_n(\Z)$.
\end{definition}

When a commutative polynomial $f(x_1,\ldots,x_k)$ is evaluated on
noncommutative polynomial expressions $P_1,\ldots,P_k$, the notation
$f(P_1,\ldots,P_k)$ uses the following fixed convention: write each scalar
monomial as $x_1^{\alpha_1}\cdots x_k^{\alpha_k}$ and replace it by
$P_1^{\alpha_1}\cdots P_k^{\alpha_k}$ in that order.  This defines a
specific noncommutative polynomial.  In the central-polynomial reductions
below, the values of the substituted expressions are scalar matrices, so their
evaluations also agree with ordinary commutative substitution.

The ambient-ring convention in Definition~\ref{def:equation} is only a
convenient notation for equations in the semiring language.  Write
\begin{equation}\label{eq:positive-negative}
 P=P^+-P^-,\qquad
 P^+,P^-\in\N\langle X_1,\ldots,X_k\rangle,
\end{equation}
by separating the positive and negative coefficients.  Then $P=0$ on
nonnegative matrices means precisely the single equation $P^+=P^-$.
Both sides use only addition, multiplication, $0$, and $1$.  In particular,
a signed polynomial expression occurring in a construction is not an
additional function symbol of the semiring and is not a quantified matrix
variable.

\begin{definition}
For $R\in\{\N,\Z,\Q\}$, let $\Eq_n(R)$ be the decision problem whose inputs
are noncommutative polynomials with scalar integer coefficients and whose
yes-instances are the equations solvable in $M_n(R)$.  Let $\Sys_n(R)$ be
the analogous problem for finite systems of such equations.  We also call
$\Sys_n(R)$ the \emph{Diophantine theory} of $M_n(R)$ in this signature.
The definitions use the same input syntax for every $n$.
For $R\in\{\Z,\Q\}$, let $\Hten(R)$ denote the problem of solvability
of one ordinary polynomial equation with integer coefficients in finitely
many variables taking values in $R$.  The number of variables is part of
the input.  Rational coefficients, if permitted, can be replaced by integer
coefficients by multiplying the equation by a common denominator.
\end{definition}

A system corresponds to a \emph{positive primitive sentence}
$\exists\mathbf X\,\bigwedge_j P_j(\mathbf X)=0$: existential
quantifiers precede a finite conjunction of equations.  A \emph{positive
existential sentence} also permits finite disjunctions; distributive laws
put it into a finite disjunction of sentences of this form.  The respective
theory is the set of sentences true in the structure.  Consequently the
inclusion results below also hold for the full positive existential
theories.  No inequalities or named noncentral matrix parameters are
included in these positive theories.  An unrestricted existential theory
allows arbitrary Boolean combinations in its quantifier-free part;
the $\exists\forall$ theory allows an existential block followed by a
universal block and a quantifier-free part.

All these are classical computability problems: inputs are finite strings
encoding coefficients and monomials.  A \emph{computable many-one reduction}
from a decision problem $A$ to a problem $B$ is a total computable map
sending an input $x$ to an input $\rho(x)$ such that
$x\in A$ if and only if $\rho(x)\in B$.  We write $A\leq_m B$.
We write $A\equiv_m B$ when both $A\leq_m B$ and $B\leq_m A$.
A decision problem is \emph{computably enumerable} (abbreviated c.e.) if an algorithm halts
and accepts precisely its yes-instances, with no requirement to halt on
other inputs.  It is \emph{decidable} if an algorithm halts on every input
and answers correctly.  A computably enumerable problem is \emph{complete}
under these reductions if every computably enumerable problem reduces to
it.  No complexity bound on the reductions is asserted.

We use two classical facts.  First, $\Hten(\Z)$ is undecidable and is
complete among computably enumerable decision problems under computable
many-one reductions; this is a standard consequence of the
Davis--Putnam--Robinson--Matiyasevich theorem
\cite{DPR1961,Matiyasevich1971,Matiyasevich1993}.
Second, Lagrange's four-square theorem gives
\begin{equation}\label{eq:four-squares}
 \N=\{z_1^2+z_2^2+z_3^2+z_4^2:z_1,z_2,z_3,z_4\in\Z\}.
\end{equation}
This is the usual device for passing between the natural-number and
integer formulations of Hilbert's tenth problem; see
\cite{Matiyasevich1993}.

We also recall the linear algebra terminology used below.  The
\emph{minimal polynomial} $\mu_A$ of a rational matrix $A$ is the monic polynomial
of least positive degree that annihilates $A$.  Division with remainder
shows that it divides every annihilating polynomial.  The
Cayley--Hamilton theorem therefore bounds its degree by the size of $A$.
For a monic polynomial $q$ of degree $d$, its \emph{companion matrix}
represents multiplication by the class of $T$ on $\Q[T]/(q)$, where
$(q)$ is the ideal generated by $q$, in the
basis $1,T,\ldots,T^{d-1}$.  It is annihilated by $q$, and it has integer
entries if $q\in\Z[T]$.  We use these facts, Eisenstein's irreducibility
criterion, and Gauss's lemma in their usual forms; see \cite{Lang2002}.

\section{The dimension and the divisibility lattice}
\label{sec:dimension}

The following one-variable equation detects precisely the multiples of a
given dimension.  Its nonnegative witnesses make the same argument work for
all three entry domains.

\begin{lemma}[A dimension detector]\label{lem:dimension-detector}
For positive integers $r,m$ and $R\in\{\N,\Z,\Q\}$,
\begin{equation}\label{eq:dimension-detector}
 (\exists A\in M_m(R))\ A^r=2I_m
 \quad\Longleftrightarrow\quad r\mid m.
\end{equation}
\end{lemma}

\begin{proof}
Let $T_r\in M_r(\N)$ be the weighted cyclic permutation matrix defined on
the standard basis by
\[
 T_re_1=2e_r,\qquad T_re_j=e_{j-1}\quad(2\leq j\leq r).
\]
For $r=1$ this means $T_1=(2)$.  In one full cycle every basis vector
acquires the factor $2$ exactly once, so $T_r^r=2I_r$.
If $m=qr$, the block diagonal sum of $q$ copies of $T_r$ is a solution in
$M_m(\N)$, hence also in $M_m(\Z)$ and $M_m(\Q)$.

Conversely, suppose $A\in M_m(\Q)$ and $A^r=2I_m$.  The polynomial
$p_r(T)=T^r-2$ is irreducible over $\Q$ by Eisenstein's criterion at $2$.
Therefore $K=\Q[T]/(p_r)$ is a field of degree $r$ over $\Q$.
Write $[h(T)]$ for the residue class of $h(T)$ modulo $(p_r)$.
The rule
\[
 [h(T)]\cdot v=h(A)v\qquad(v\in\Q^m)
\]
makes $\Q^m$ a vector space over $K$: it is well defined because $p_r(A)=0$,
and the identity acts as the identity.  It is finite dimensional over $K$
because a rational basis also spans it over $K$.  Comparing dimensions gives
\[
 m=[K:\Q]\dim_K(\Q^m)=r\dim_K(\Q^m),
\]
where $[K:\Q]=\dim_{\Q}K$ is the degree of the field extension.
Hence $r\mid m$.  This also proves necessity for either smaller entry domain.
The algebraic facts used here are standard; see \cite{Lang2002}.
\end{proof}

\begin{theorem}[Exact dimension comparison]\label{thm:dimension-lattice}
Let $R\in\{\N,\Z,\Q\}$ and let $n,m\geq1$.  Then
\begin{equation}\label{eq:dimension-lattice}
 \Eq_n(R)\subseteq\Eq_m(R)
 \quad\Longleftrightarrow\quad
 \Sys_n(R)\subseteq\Sys_m(R)
 \quad\Longleftrightarrow\quad n\mid m.
\end{equation}
Either inclusion is strict exactly when $n$ is a proper divisor of $m$.
Equality of either pair of theories holds exactly when $n=m$.
\end{theorem}

\begin{proof}
If $m=qn$, the map
\[
 \delta:M_n(R)\longrightarrow M_m(R),\qquad
 \delta(A)=\operatorname{diag}(A,\ldots,A)
\]
with $q$ copies is an injective homomorphism preserving $0$ and $1$.
In particular, $\delta(aI_n)=aI_m$ for every allowed coefficient.
For every noncommutative polynomial $P$,
\[
 P(\delta(A_1),\ldots,\delta(A_k))
 =\delta(P(A_1,\ldots,A_k)).
\]
For signed coefficients on nonnegative matrices this equality is understood
in the ambient integer rings.  Applying $\delta$ to a solution proves both
inclusions in \eqref{eq:dimension-lattice}.

Conversely, the equation $X^n-2=0$ belongs to $\Eq_n(R)$ and, as a
one-equation system, to $\Sys_n(R)$.  Either asserted inclusion forces a
solution in dimension $m$, so Lemma~\ref{lem:dimension-detector} gives
$n\mid m$.  If $n\mid m$ and $n<m$, the equation $X^m-2=0$ is solvable
in dimension $m$ and not in dimension $n$.  Thus both inclusions are strict.
The equality statements follow.
\end{proof}

\begin{corollary}\label{cor:lattice}
For each of the three domains, the families $(\Eq_n(R))_{n\geq1}$ and
$(\Sys_n(R))_{n\geq1}$, with their induced inclusion orders, are lattices
isomorphic to the positive integers ordered by divisibility.  Their internal
meet and join correspond to the greatest common divisor $\gcd$ and
the least common multiple $\operatorname{lcm}$.
For example,
\[
 \Eq_1(R)\subsetneq\Eq_2(R)\subsetneq\Eq_4(R)\subsetneq\cdots,
\]
while $\Eq_2(R)$ and $\Eq_3(R)$ are incomparable.  The same statements
hold for systems and for positive existential theories.
\end{corollary}

\begin{remark}\label{rem:lattice-not-sets}
The meet and join are taken within the displayed families.  They need not
be set intersection and set union.  For example,
$(X^2-2)(X^3-2)=0$ belongs to both $\Eq_2(R)$ and $\Eq_3(R)$, but not
to $\Eq_1(R)$, since it has no rational scalar root.  Also $X^6-2=0$
belongs to $\Eq_6(R)$ but to neither $\Eq_2(R)$ nor $\Eq_3(R)$.
\end{remark}

\section{Central polynomials and an integral range lemma}
\label{sec:central}

\begin{definition}
A polynomial $c\in\Q\langle T_1,\ldots,T_r\rangle$ is a
\emph{central polynomial for $M_n(\Q)$} if every value of $c$ is in
$\Q I_n$ and at least one value is nonzero.  Here a polynomial is called
\emph{fully multilinear} if it has the form
\begin{equation}\label{eq:multilinear}
 c(T_1,\ldots,T_r)
   =\sum_{\sigma\in S_r}b_\sigma
      T_{\sigma(1)}\cdots T_{\sigma(r)}.
\end{equation}
Thus every variable occurs once in every monomial with nonzero coefficient.
\end{definition}

\begin{theorem}[Classical central-polynomial theorem]
\label{thm:classical-central}
For every $n\geq1$ there are $r\geq1$ and a fully multilinear polynomial
\[
 c\in\free{T_1,\ldots,T_r}
\]
which is central for $M_n(\Q)$.
\end{theorem}

\begin{proof}[Reference and coefficient convention]
Razmyslov \cite[Theorem~1 and p.~486]{Razmyslov1973} constructs a nonidentity
multilinear central polynomial of positive degree.  Apply his theorem over
$\Q$ and multiply by a common denominator of its coefficients.  This
preserves multilinearity, centrality, and nonvanishing.  For $n=1$, one may
simply take $c(T)=T$.  Formanek's independent construction is given in
\cite{Formanek1972}; see also the discussion in
\cite[Introduction]{BresarDrensky2013}.
\end{proof}

\begin{definition}[B\'ezout polynomials]\label{def:bezout-polynomials}
Fix $n\geq1$ and a fully multilinear central polynomial
$c\in\free{T_1,\ldots,T_r}$, with $r\geq1$.
Let $\mathcal B=\{E_{ij}:1\leq i,j\leq n\}$ and enumerate all its
ordered $r$-tuples as
\[
 \tau_j=(B_{j,1},\ldots,B_{j,r}),\qquad
 1\leq j\leq N_c,\qquad N_c=(n^2)^r.
\]
Write their scalar values as
\begin{equation}\label{eq:basis-values}
 c(\tau_j)=\beta_j I_n,\qquad \beta_j\in\Z.
\end{equation}
At least one $\beta_j$ is nonzero: otherwise expansion in the elementary matrix
basis, using multilinearity, would make $c$ vanish at every rational
matrix tuple.  Consequently the greatest common divisor
$d=\gcd(\beta_1,\ldots,\beta_{N_c})$ is taken to be positive.
A \emph{B\'ezout coefficient tuple} for these values is a tuple
$\boldsymbol\lambda=(\lambda_1,\ldots,\lambda_{N_c})\in\Z^{N_c}$
satisfying the ordinary integer B\'ezout identity
\begin{equation}\label{eq:bezout}
 \sum_{j=1}^{N_c}\lambda_j\beta_j=d.
\end{equation}
Such a tuple exists and can be computed by the extended Euclidean algorithm.

Introduce two collections
$\mathbf U=(U_{j,\nu})$ and $\mathbf V=(V_{j,\nu})$, where
$1\leq j\leq N_c$ and $1\leq\nu\leq r$.
All $2rN_c$ variable symbols are distinct; in particular, the blocks
used by the different copies of $c$ have pairwise disjoint variable sets.
We call
\begin{equation}\label{eq:G-definition}
 G_{c,\boldsymbol\lambda}(\mathbf U)
 =\sum_{j=1}^{N_c}\lambda_j c(U_{j,1},\ldots,U_{j,r})
\end{equation}
the \emph{B\'ezout combination} associated with $c$ and
$\boldsymbol\lambda$, and call
\begin{equation}\label{eq:Q-definition}
 Q_{c,\boldsymbol\lambda}(\mathbf U,\mathbf V)
 =G_{c,\boldsymbol\lambda}(\mathbf U)
  -G_{c,\boldsymbol\lambda}(\mathbf V)
\end{equation}
the associated \emph{B\'ezout range polynomial}.
These are the B\'ezout polynomials referred to in this article.
The terminology records their construction from \eqref{eq:bezout};
it does not specify a unique polynomial independently of the choices of
$c$, the enumeration, and the coefficients $\boldsymbol\lambda$.
When these choices are fixed, we abbreviate the two expressions by
$G$ and $Q_n$, respectively.  Both are noncommutative polynomials with
scalar integer coefficients, and every monomial in either expression
has degree $r$.  These polynomials are generally not fully multilinear
in their entire declared variable lists.
\end{definition}

We now establish the range property of this construction.  Its number of
matrix arguments depends only on the fixed dimension and the chosen
central polynomial.

\begin{proposition}[Exact integral range]
\label{prop:exact-range}
For every $n\geq1$ there are a positive integer $d_n$, an integer $s_n\geq1$,
and a homogeneous polynomial of positive degree
\[
 Q_n\in\free{Z_1,\ldots,Z_{s_n}}
\]
such that
\begin{equation}\label{eq:exact-range}
 Q_n\bigl(M_n(\N)^{s_n}\bigr)
 =Q_n\bigl(M_n(\Z)^{s_n}\bigr)
 =d_n\Z I_n.
\end{equation}
Both image sets in \eqref{eq:exact-range} are evaluated in $M_n(\Z)$.
\end{proposition}

\begin{proof}
Choose $c$ by Theorem~\ref{thm:classical-central}, and use the notation
and B\'ezout polynomials of Definition~\ref{def:bezout-polynomials}.
There are two assertions to prove: every value is a multiple of $dI_n$,
and every such multiple is attained with nonnegative arguments.

First, for integral matrices $X_1,\ldots,X_r$, expand
$X_\nu=\sum_{B\in\mathcal B}x_{\nu,B}B$, where
$x_{\nu,B}\in\Z$.  Multilinearity gives
\begin{equation}\label{eq:basis-expansion}
 c(X_1,\ldots,X_r)=
 \sum_{(B_1,\ldots,B_r)\in\mathcal B^r}
 \left(\prod_{\nu=1}^r x_{\nu,B_\nu}\right)c(B_1,\ldots,B_r).
\end{equation}
Each final factor is one of the values $\beta_jI_n$, and $d$ divides
every $\beta_j$.  Hence
\begin{equation}\label{eq:divisibility}
 c\bigl(M_n(\Z)^r\bigr)\subseteq d\Z I_n.
\end{equation}
Integer linear combinations and differences preserve this containment.
Thus every integral value of $G$, and of $Q_n=G(\mathbf U)-G(\mathbf V)$,
lies in $d\Z I_n$.

Next, let $b\in\N$.  For each block of $G$, assign
\begin{equation}\label{eq:G-witnesses}
 U_{j,1}=bB_{j,1},\qquad
 U_{j,\nu}=B_{j,\nu}\quad(2\leq\nu\leq r),
 \qquad 1\leq j\leq N_c.
\end{equation}
All assigned matrices have nonnegative integer entries.  Linearity in
the first argument of $c$ gives
\[
 c(U_{j,1},\ldots,U_{j,r})
 =b\,c(B_{j,1},\ldots,B_{j,r})=b\beta_jI_n.
\]
Therefore the B\'ezout identity makes their weighted sum equal to
\[
 G(\mathbf U)=\sum_{j=1}^{N_c}\lambda_jb\beta_jI_n
 =b\left(\sum_{j=1}^{N_c}\lambda_j\beta_j\right)I_n=bdI_n.
\]

For an arbitrary $a\in\Z$, choose
$b_+=\max(a,0)$ and $b_-=\max(-a,0)$, so that $a=b_+-b_-$.
Use \eqref{eq:G-witnesses} with multiplier $b_+$ in the $\mathbf U$
blocks and with multiplier $b_-$ in the $\mathbf V$ blocks.
Their variable sets are disjoint, so these assignments can be made
simultaneously.  They yield
\[
 Q_n(\mathbf U,\mathbf V)
 =b_+dI_n-b_-dI_n=adI_n.
\]
This proves surjectivity on nonnegative arguments, and therefore also
on integral arguments.  Take $d_n=d$ and $s_n=2rN_c$.
Every monomial has degree $r\geq1$, and the attained value $dI_n\ne0$
shows that $Q_n$ is not the zero polynomial.  Its degree is therefore
exactly $r$.
\end{proof}

\begin{remark}\label{rem:no-parameters}
The elementary matrices in this proof are witnesses used in a finite arithmetic
calculation and in the proof of surjectivity.  The polynomial $Q_n$ itself
contains only variables and scalar integer coefficients.  In particular,
there is no elementary matrix parameter and no equation requiring any variable to
equal an elementary matrix.
\end{remark}

\begin{remark}[Choice of parameters]\label{rem:choices}
The parameters $(d_n,s_n)$ depend on the chosen polynomial $Q_n$ in
\eqref{eq:exact-range}; they are not canonical invariants of the matrix
algebra.  The preceding construction gives $s_n=2rN_c$, where
$N_c=(n^2)^r$, with $d_n$ its basis-value gcd.  The integral reductions
allow any uniformly effective choice with this range; the rational
reductions use the homogeneous choices constructed here.
Appendix~\ref{app:parameters} gives explicit alternative choices, including
$s_1=s_2=2$ and $s_n=3$ for $n\geq3$, and tabulates their scale factors.
\end{remark}

\begin{remark}\label{rem:dividing}
The proof does not assert that one can divide the coefficients of $c$ or
$Q_n$ by $d_n$.  Universal divisibility of values need not imply coefficient
divisibility.  The reduction in Section~\ref{sec:main} instead clears a
scalar denominator after the polynomial substitutions have been specified.
\end{remark}

\begin{lemma}[Uniform effectivity]\label{lem:effective}
The objects $d_n,s_n,Q_n$ in Proposition~\ref{prop:exact-range} can be
constructed by an algorithm from $n$.
\end{lemma}

\begin{proof}
Enumerate, in a single effective enumeration, all positive integers $r$
together with all integer coefficient vectors $(b_\sigma)_{\sigma\in S_r}$
in \eqref{eq:multilinear}.  For each candidate, evaluate the polynomial on
all tuples from $\mathcal B^r$.  Accept a candidate if every resulting
matrix is scalar and at least one is nonzero.

The finite test is correct.  Equation~\eqref{eq:basis-expansion}, with
rational coefficients, shows that scalarity on all basis tuples implies
scalarity on all rational matrix tuples.  The same expansion shows that
vanishing on all basis tuples is equivalent to being a polynomial identity.
Theorem~\ref{thm:classical-central} guarantees that the search eventually
accepts a candidate.  All subsequent steps in the preceding proof are
finite matrix calculations, a gcd calculation, and the extended Euclidean
algorithm.  They construct $d_n,s_n,Q_n$ effectively.
\end{proof}

\section{Reduction to one matrix equation}
\label{sec:main}

\begin{theorem}[A common instance for the two matrix domains]
\label{thm:main-reduction}
There is an algorithm which, given $n\geq1$ and a polynomial
$f\in\Z[x_1,\ldots,x_k]$, constructs one polynomial
\[
 F_{n,f}\in\free{Z_1,\ldots,Z_\ell}
\]
with the following equivalences:
\begin{equation}\label{eq:main-equivalence}
\begin{aligned}
 (\exists\mathbf a\in\Z^k)\ f(\mathbf a)=0
 &\iff(\exists\mathbf A\in M_n(\N)^\ell)\ F_{n,f}(\mathbf A)=0\\
 &\iff(\exists\mathbf A\in M_n(\Z)^\ell)\ F_{n,f}(\mathbf A)=0.
\end{aligned}
\end{equation}
For a fixed $n$, the number of matrix variables can be taken to be $\ell=ks_n$,
where $s_n$ is the constant in Proposition~\ref{prop:exact-range}.
\end{theorem}

\begin{proof}
Choose $d=d_n$, $s_n$, and $Q_n$ effectively as in
Lemma~\ref{lem:effective}, or use the explicit family in
Appendix~\ref{app:parameters}.
Write the polynomial as a finite sum over multiindices
$\alpha=(\alpha_1,\ldots,\alpha_k)\in\N^k$:
\begin{equation}\label{eq:scalar-polynomial}
 f(x_1,\ldots,x_k)=\sum_\alpha b_\alpha
 x_1^{\alpha_1}\cdots x_k^{\alpha_k},
 \qquad |\alpha|=\alpha_1+\cdots+\alpha_k.
\end{equation}
Choose a nonnegative integer $D$ at least as large as every $|\alpha|$
occurring in the sum.  Take $D=0$ for constant polynomials, including the
zero polynomial.  For $1\leq i\leq k$, introduce a block
$\mathbf Z^{(i)}=(Z_{i,1},\ldots,Z_{i,s_n})$, with the $k$ variable
sets pairwise disjoint, and write
\[
 Q_n^{(i)}(\mathbf Z^{(i)})
 :=Q_n(Z_{i,1},\ldots,Z_{i,s_n}).
\]
These are copies of the same polynomial on separately named variables.
Their blocks contain $\ell=ks_n$ variables altogether.  Suppressing the
arguments, define
\begin{equation}\label{eq:main-polynomial}
 \boxed{
 F_{n,f}=\sum_\alpha b_\alpha d^{D-|\alpha|}
 \bigl(Q_n^{(1)}\bigr)^{\alpha_1}\cdots\bigl(Q_n^{(k)}\bigr)^{\alpha_k}.}
\end{equation}
The copies $Q_n^{(i)}$ are polynomial abbreviations and are expanded when the
output polynomial is written.  Every exponent of $d$ is nonnegative, so
\eqref{eq:main-polynomial} has integer coefficients.  Although the $Q_n^{(i)}$
are written in a fixed order, their values will commute because they are
scalar matrices.

For any assignment of integral matrix arguments, Proposition~\ref{prop:exact-range}
provides uniquely determined integers $a_i$ such that, at the chosen
assignment $\mathbf A^{(i)}$ to each block,
\[
 Q_n^{(i)}(\mathbf A^{(i)})=da_iI_n\qquad(1\leq i\leq k).
\]
Substitution in \eqref{eq:main-polynomial} gives
\begin{equation}\label{eq:evaluation-identity}
 F_{n,f}(\mathbf A^{(1)},\ldots,\mathbf A^{(k)})
 =d^D f(a_1,\ldots,a_k)I_n.
\end{equation}
Since $d>0$, a matrix zero of $F_{n,f}$ yields an integer zero of $f$.
This proves soundness for either matrix domain.

Conversely, let $(a_1,\ldots,a_k)$ be an integer zero of $f$.
By Proposition~\ref{prop:exact-range}, each value $da_iI_n$ can be realized
by $Q_n^{(i)}$ using nonnegative matrix arguments.  The variable blocks are
disjoint, so these witnesses can be chosen independently.  Equation
\eqref{eq:evaluation-identity} then gives $F_{n,f}=0$.  The same witnesses
are also integral matrices, proving both converses.

All operations in the construction are effective, including expansion of
the resulting finite polynomial.  No bound on the size of that expansion
is needed for computable reducibility.
\end{proof}

\begin{corollary}[Single-equation undecidability]\label{cor:undecidability}
For every fixed $n\geq1$, both $\Eq_n(\N)$ and $\Eq_n(\Z)$ are undecidable.
Equivalently, Hilbert's tenth problem has a negative answer for single
noncommutative polynomial equations with scalar integer coefficients over
each of $M_n(\N)$ and $M_n(\Z)$.
\end{corollary}

\begin{proof}
Theorem~\ref{thm:main-reduction} gives
$\Hten(\Z)\leq_m\Eq_n(R)$ for either $R$.  Apply the undecidability of
$\Hten(\Z)$.
\end{proof}

\begin{remark}[The natural-number syntax]\label{rem:semiring-syntax}
In the proof for $M_n(\N)$, the expressions $c$, $G$, and $Q_n^{(i)}$ may have
negative entries when evaluated.  They are signed polynomial expressions
evaluated in $M_n(\Z)$, not extra variables required to be nonnegative.
After expanding \eqref{eq:main-polynomial}, write
$F_{n,f}=F_{n,f}^{+}-F_{n,f}^{-}$.  Its zero equation is the single
semiring equation
\[
 F_{n,f}^{+}=F_{n,f}^{-}
\]
in the original signature $\{+,\cdot,0,1\}$.  Every quantified matrix
argument remains in $M_n(\N)$.
\end{remark}

For completeness, a direct parametrization of nonnegative scalar values
also follows from the same construction.

\begin{corollary}[A scaled natural-number range]\label{cor:natural-range}
For each $n$ there is an effectively constructible polynomial $H_n$ such
that, on either matrix domain,
\begin{equation}\label{eq:natural-range}
 H_n\bigl(M_n(R)^{4s_n}\bigr)=d_n^2\N I_n,
 \qquad R\in\{\N,\Z\}.
\end{equation}
Consequently, solvability of a scalar equation over $\N$ reduces directly
to solvability of one equation over $M_n(\N)$.
\end{corollary}

\begin{proof}
Take four copies $Q_n^{(1)},\ldots,Q_n^{(4)}$ on pairwise disjoint
sets of $s_n$ variables, and set
\[
 H_n=\sum_{j=1}^4\bigl(Q_n^{(j)}\bigr)^2.
\]
Every value is $d_n^2(a_1^2+a_2^2+a_3^2+a_4^2)I_n$ for integers $a_i$.
Conversely, \eqref{eq:four-squares} and Proposition~\ref{prop:exact-range}
realize every element of $d_n^2\N I_n$ with nonnegative matrix arguments.
For the final assertion, use the construction
\eqref{eq:main-polynomial} with copies $H_n^{(i)}$, each on a
separate block of $4s_n$ variables, and with $d_n^2$ in place of $d_n$.
The variable sets of these blocks are pairwise disjoint.
\end{proof}

\section{Explicit polynomials in dimension two}
\label{sec:dimension-two}

In dimension two, a small central polynomial already has the full integral
scalar range on nonnegative matrix arguments.  Write $[A,B]=AB-BA$.

\begin{proposition}\label{prop:two}
The fully multilinear polynomial
\begin{equation}\label{eq:C-two}
 C_2(X_1,X_2,X_3,X_4)
 =[X_1,X_2][X_3,X_4]+[X_3,X_4][X_1,X_2]
\end{equation}
has the exact ranges
\begin{equation}\label{eq:C-two-range}
 C_2\bigl(M_2(\N)^4\bigr)
 =C_2\bigl(M_2(\Z)^4\bigr)
 =\Z I_2.
\end{equation}
\end{proposition}

\begin{proof}
For traceless matrices
\[
 A=\begin{pmatrix}a&b\\c&-a\end{pmatrix},\qquad
 B=\begin{pmatrix}d&e\\f&-d\end{pmatrix},
\]
direct multiplication gives
\begin{equation}\label{eq:traceless-two}
 AB+BA=(2ad+bf+ce)I_2=\tr(AB)I_2.
\end{equation}
Both commutators in \eqref{eq:C-two} are traceless.  Hence every integral
value of $C_2$ is an integer scalar matrix.

For every $t\in\N$, the elementary matrix multiplication rules give
\begin{align}
 C_2(E_{11},E_{12},tE_{21},E_{11})&=tI_2,
 \label{eq:two-positive}\\
 C_2(E_{11},E_{12},E_{11},tE_{21})&=-tI_2.
 \label{eq:two-negative}
\end{align}
Indeed, $[E_{11},E_{12}]=E_{12}$,
$[tE_{21},E_{11}]=tE_{21}$, and
$[E_{11},tE_{21}]=-tE_{21}$, while
$E_{12}E_{21}+E_{21}E_{12}=I_2$.
All arguments in \eqref{eq:two-positive}--\eqref{eq:two-negative} are
nonnegative matrices.  Thus all integer scalar values occur on both domains.
\end{proof}

\begin{corollary}\label{cor:two-reduction}
Let $f\in\Z[x_1,\ldots,x_k]$.  For $1\leq i\leq k$, let
\[
 C_2^{(i)}=C_2(Z_{i,1},Z_{i,2},Z_{i,3},Z_{i,4}),
\]
where the $k$ sets of four variable symbols are pairwise disjoint.  Then
\begin{equation}\label{eq:two-equation}
 f(C_2^{(1)},\ldots,C_2^{(k)})=0
\end{equation}
is a single equation in $4k$ matrix variables, solvable over either
$M_2(\N)$ or $M_2(\Z)$ if and only if $f=0$ has an integer solution.
\end{corollary}

\begin{proof}
Every evaluation of $C_2^{(i)}$ is $a_iI_2$ for an integer $a_i$, and every such integer
tuple is realizable independently using nonnegative arguments.  Thus the
left side of \eqref{eq:two-equation} equals $f(a_1,\ldots,a_k)I_2$.
\end{proof}

The polynomial in \eqref{eq:C-two} can therefore be used directly in place
of $Q_2$, with $(d_2,s_2)=(1,4)$.  Its displayed elementary matrices occur only in the proof
of its range.  In dimension one the polynomial $Q_1(U,V)=U-V$ gives the
same range statement on both scalar domains.

If full multilinearity is not required, two variables suffice in
dimension two.

\begin{proposition}[A two-variable polynomial]\label{prop:two-variables}
The polynomial
\begin{equation}\label{eq:B-two}
 B_2(X,Y)=[X,Y]^2
\end{equation}
satisfies
\begin{equation}\label{eq:B-two-range}
 B_2\bigl(M_2(\N)^2\bigr)
 =B_2\bigl(M_2(\Z)^2\bigr)=\Z I_2.
\end{equation}
Thus $(d_2,s_2)=(1,2)$ is an admissible choice in
Proposition~\ref{prop:exact-range}.
\end{proposition}

\begin{proof}
The commutator of two $2\times2$ matrices has trace zero.  The
Cayley--Hamilton theorem gives
$[X,Y]^2=-\det([X,Y])I_2$, proving that every integral value is an integer
scalar matrix.  For $u,v\in\N$, take
\[
 X_0=\begin{pmatrix}1&1\\0&0\end{pmatrix},\qquad
 Y(u,v)=\begin{pmatrix}u&v\\1&0\end{pmatrix}.
\]
Direct multiplication gives
\[
 [X_0,Y(u,v)]=\begin{pmatrix}1&v-u\\-1&-1\end{pmatrix},
 \qquad B_2(X_0,Y(u,v))=(1+u-v)I_2.
\]
For a prescribed $a\in\Z$, choose
$u=\max(a-1,0)$ and $v=\max(1-a,0)$.
This realizes $aI_2$ with nonnegative matrix arguments.
\end{proof}

Consequently the equation
$f(B_2^{(1)},\ldots,B_2^{(k)})=0$, with
$B_2^{(i)}=B_2(Z_{i,1},Z_{i,2})$ on pairwise disjoint sets of two variables,
gives the reduction of Corollary~\ref{cor:two-reduction} using $2k$
matrix variables.  The four-variable version remains a fully multilinear
example.

\begin{remark}[Two numerical evaluations]\label{rem:commutator-numerics}
For the value $2I_2$, one may take
\[
 X=\begin{pmatrix}1&1\\0&0\end{pmatrix},\qquad
 Y=\begin{pmatrix}1&0\\1&0\end{pmatrix}.
\]
Then
\[
 XY=\begin{pmatrix}2&0\\0&0\end{pmatrix},\qquad
 YX=\begin{pmatrix}1&1\\1&1\end{pmatrix},\qquad
 [X,Y]=\begin{pmatrix}1&-1\\-1&-1\end{pmatrix},
\]
and
\[
 [X,Y]^2=
 \begin{pmatrix}1+1&-1+1\\-1+1&1+1\end{pmatrix}=2I_2.
\]
For the two matrices with consecutive entries,
\[
 X=\begin{pmatrix}1&2\\3&4\end{pmatrix},\qquad
 Y=\begin{pmatrix}5&6\\7&8\end{pmatrix},
\]
the corresponding calculation is
\[
 XY=\begin{pmatrix}19&22\\43&50\end{pmatrix},\qquad
 YX=\begin{pmatrix}23&34\\31&46\end{pmatrix},\qquad
 [X,Y]=\begin{pmatrix}-4&-12\\12&4\end{pmatrix},
\]
so that
\[
 [X,Y]^2=
 \begin{pmatrix}16-144&48-48\\-48+48&-144+16\end{pmatrix}
 =-128I_2.
\]
The range is not restricted to even multiples: taking $u=v=0$ in the
preceding proof gives $I_2$.
\end{remark}

\section{Finite systems and computability consequences}
\label{sec:systems}

The single-equation result also gives a solvability-preserving reduction
from arbitrary finite matrix systems.  We first record the scalar reduction
used for the reverse direction.

\begin{lemma}[Entrywise reduction]\label{lem:entrywise}
For every $n\geq1$ and $R\in\{\N,\Z\}$, solvability of a finite system
over $M_n(R)$ reduces effectively to solvability of one scalar polynomial
equation over $\Z$.  This construction is uniform in $n$.
\end{lemma}

\begin{proof}
Consider a finite system
\[
 P_1(X_1,\ldots,X_k)=0,\quad\ldots,\quad P_q(X_1,\ldots,X_k)=0.
\]
Replace each matrix variable by an $n\times n$ matrix of scalar variables,
and perform the indicated additions and multiplications symbolically.
Each matrix equation yields $n^2$ commutative polynomial equations in the
scalar entry variables.  Denote the resulting list by
\[
 p_1(\mathbf x)=0,\ldots,p_{qn^2}(\mathbf x)=0.
\]

If $R=\Z$, this scalar system is equivalent to
\begin{equation}\label{eq:sum-residual-squares}
 h(\mathbf x)=\sum_{j=1}^{qn^2}p_j(\mathbf x)^2=0,
\end{equation}
because the summands are squares of integers.

If $R=\N$, replace each scalar entry variable $x$ by
$y_{x,1}^2+y_{x,2}^2+y_{x,3}^2+y_{x,4}^2$, with a separate group of four integer variables for each entry.
These groups have pairwise disjoint variable sets.  By \eqref{eq:four-squares}, this substitution ranges
over exactly all nonnegative entry assignments.  Let $\widetilde p_j$
be the resulting polynomials.  The original matrix system is solvable
if and only if
\begin{equation}\label{eq:natural-residual-squares}
 h(\mathbf y)=\sum_{j=1}^{qn^2}\widetilde p_j(\mathbf y)^2=0
\end{equation}
has an integer solution.  Every step is a finite effective polynomial
calculation.
\end{proof}

\begin{theorem}[From systems to one equation]\label{thm:systems-one}
For every fixed $n\geq1$ and $R\in\{\N,\Z\}$, there is an effective
transformation from a finite system of noncommutative polynomial equations
with scalar integer coefficients to one such equation over $M_n(R)$,
preserving existence of solutions.
\end{theorem}

\begin{proof}
Apply Lemma~\ref{lem:entrywise} to obtain a scalar polynomial $h$ whose
integer solvability is equivalent to that of the given matrix system.
Then apply Theorem~\ref{thm:main-reduction} to $h$, with the same dimension
$n$.  The resulting single equation is equivalent to the original system
over the specified matrix domain.
\end{proof}

\begin{remark}\label{rem:scope-conjunction}
Theorem~\ref{thm:systems-one} is a reduction of solvability problems.  Its
new equation generally has different variables from the original system.
It does not assert that a conjunction in prescribed free matrix variables
has been replaced by one equation defining the same set in those original
variables.  All uses of sums of squares in
Lemma~\ref{lem:entrywise} concern scalar integers.
\end{remark}

\begin{corollary}[Many-one equivalence and completeness]
\label{cor:completeness}
For every fixed $n\geq1$ and $R\in\{\N,\Z\}$,
\begin{equation}\label{eq:many-one}
 \Eq_n(R)\equiv_m\Sys_n(R)\equiv_m\Hten(\Z).
\end{equation}
In particular, $\Eq_n(R)$ and $\Sys_n(R)$ are complete among computably
enumerable decision problems under computable many-one reductions.
\end{corollary}

\begin{proof}
The lower reduction from $\Hten(\Z)$ is
Theorem~\ref{thm:main-reduction}; the upper reduction is
Lemma~\ref{lem:entrywise}.  A single equation is a one-equation system, so
the remaining equivalence also follows, or directly follows from
Theorem~\ref{thm:systems-one}.

For clarity, these matrix problems are computably enumerable: enumerate
finite tuples of matrices with entries of increasing absolute-value bound
(or increasing nonnegative bound), and evaluate the input equations
exactly.  A solution is eventually found whenever one exists.  Completeness
now follows from that of $\Hten(\Z)$.
\end{proof}

The effective construction supplies single equations in the original
matrix signature.  The algebraic existence theorem is classical; the
integral range calculation ensures that its application introduces neither
noncentral coefficients nor auxiliary scalarity constraints.

\begin{corollary}\label{cor:larger-languages}
For each fixed $n$, single-equation solvability remains undecidable if
additional matrix parameters are permitted, and both the existential and
the $\exists\forall$ theories of $M_n(\N)$ and $M_n(\Z)$ are undecidable.
\end{corollary}

\begin{proof}
Each expanded problem contains the scalar-coefficient single-equation
problem as a subclass of its inputs.  For the $\exists\forall$ prefix,
one may add a universally quantified variable that does not occur in the
equation.
\end{proof}

\section{Rational matrix equations}
\label{sec:rational}

Write $\Q_{\geq0}=\{q\in\Q:q\geq0\}$.
Polynomial expressions on $M_n(\Q_{\geq0})$ are evaluated in $M_n(\Q)$,
with the same signed-polynomial convention as in
Definition~\ref{def:equation}.

\subsection{Full rational scalar ranges}

We first explain why a fully multilinear central polynomial can produce
every rational scalar.  We then treat the homogeneous range polynomials
$Q_n$, including their explicit alternatives.  Two scaling rules will be
used: scaling one argument of a fully multilinear polynomial scales its
value by the same factor, whereas scaling every argument of a homogeneous
polynomial of degree $h$ scales its value by the $h$th power of that factor.

\begin{lemma}\label{lem:rational-multilinear-range}
Let $c\in\Q\langle T_1,\ldots,T_r\rangle$ be a fully multilinear
central polynomial for $M_n(\Q)$, with $r\geq1$.  Then
\[
 c\bigl(M_n(\Q)^r\bigr)=\Q I_n.
\]
If $n\geq2$, one also has
\begin{equation}\label{eq:rational-multilinear-range}
 c\bigl(M_n(\Q_{\geq0})^r\bigr)=\Q I_n.
\end{equation}
\end{lemma}

\begin{proof}
Centrality already gives containment of both image sets in $\Q I_n$.
Because $c$ is not a polynomial identity, there are rational matrices
$A_1,\ldots,A_r$ with
$c(A_1,\ldots,A_r)=\beta I_n$ for some $\beta\ne0$.
For a prescribed $q\in\Q$, linearity in the first argument gives
\[
 c\bigl((q/\beta)A_1,A_2,\ldots,A_r\bigr)
   =\frac q\beta c(A_1,\ldots,A_r)=qI_n.
\]
This proves the first assertion.  Only the first argument has been scaled.

To keep all arguments nonnegative, we need a nonnegative evaluation with
the same sign as the desired scalar.  Let
$\mathcal B=\{E_{ab}:1\leq a,b\leq n\}$ and, for each
$\tau\in\mathcal B^r$, write $c(\tau)=\beta_\tau I_n$.
At least one $\beta_\tau$ is nonzero, by expansion in the elementary matrix
basis.  We claim that positive and negative $\beta_\tau$ both occur.

For this purpose put $J=\sum_{a,b}E_{ab}$, the all-ones matrix.
Since $J^2=nJ$, one has $J^r=n^{r-1}J$ for every $r\geq1$.
Every word of $c$ therefore becomes a multiple of $J$ at
$(J,\ldots,J)$.  Its value is also a scalar matrix, by centrality.
When $n\geq2$, the only multiple of $J$ that is a scalar matrix is
the zero matrix: the off-diagonal entries force its multiplier to be
zero.  Hence $c(J,\ldots,J)=0$.  Expanding each copy of
$J=\sum_{a,b}E_{ab}$ by multilinearity gives
\[
 0=c(J,\ldots,J)
   =\sum_{\tau\in\mathcal B^r}c(\tau)
   =\left(\sum_{\tau\in\mathcal B^r}\beta_\tau\right)I_n.
\]
Thus the $\beta_\tau$ sum to zero and are not all zero, which proves
the claim about their signs.

Now let $q\ne0$.  Choose a tuple
$\tau=(B_1,\ldots,B_r)$ for which $\beta_\tau$ has the same sign
as $q$.  The rational number $q/\beta_\tau$ is positive, so
\[
 c\bigl((q/\beta_\tau)B_1,B_2,\ldots,B_r\bigr)=qI_n
\]
uses only nonnegative rational matrices.  For $q=0$, set the first
argument equal to zero.  This proves \eqref{eq:rational-multilinear-range}.
\end{proof}

The restriction $n\geq2$ in \eqref{eq:rational-multilinear-range} is
necessary: for $n=1$, the central polynomial $c(T)=T$ takes only
nonnegative values on nonnegative arguments.  The chosen polynomials
$Q_n$ have the full rational scalar range in every dimension.  The next
lemma explains how their exact integral range and homogeneity imply this
statement.

\begin{lemma}\label{lem:rational-homogeneous-range}
Suppose $P\in\free{Z_1,\ldots,Z_s}$ is homogeneous of positive degree
$h$ and, for a positive integer $d$,
\[
 P\bigl(M_n(\N)^s\bigr)=P\bigl(M_n(\Z)^s\bigr)=d\Z I_n.
\]
Then
\[
 P\bigl(M_n(\Q_{\geq0})^s\bigr)
 =P\bigl(M_n(\Q)^s\bigr)=\Q I_n.
\]
\end{lemma}

\begin{proof}
For a tuple $\mathbf A=(A_1,\ldots,A_s)$, write
$\lambda\mathbf A=(\lambda A_1,\ldots,\lambda A_s)$.
Homogeneity means that
\[
 P(\lambda\mathbf A)=\lambda^hP(\mathbf A)
 \qquad(\lambda\in\Q),
\]
because each monomial has exactly $h$ factors.

First consider arbitrary rational arguments $\mathbf A$.
Choose a positive common denominator $D$ for all their matrix entries.
Then $D\mathbf A$ is integral, so the assumed integral range and
homogeneity give
\[
 P(\mathbf A)=\frac{1}{D^h}P(D\mathbf A)
       \in\frac{d}{D^h}\Z I_n\subseteq\Q I_n.
\]
Thus all rational values are scalar, including those obtained from
nonnegative rational arguments.

For surjectivity, fix $q=a/b$ with $a\in\Z$ and $b\geq1$.
The integral range supplies nonnegative integer arguments with value
$dtI_n$ for every integer $t$.  If all these arguments are divided by
a positive integer $L$, their value becomes $(dt/L^h)I_n$.
We therefore choose $L$ and $t$ so that $dt/L^h=a/b$.
The choice
\[
 L=db,\qquad
 t=a\,d^{h-1}b^{h-1}=aL^{h-1}
\]
makes $t$ an integer, including when $a$ is negative.  Choose
$\mathbf U\in M_n(\N)^s$ with $P(\mathbf U)=dtI_n$.
Then $\mathbf U/L$ consists of nonnegative rational matrices, and
\[
 \begin{aligned}
 P(\mathbf U/L)
   &=\frac{1}{L^h}P(\mathbf U)
     =\frac{dt}{L^h}I_n\\
   &=\frac{d\,a\,(db)^{h-1}}{(db)^h}I_n
     =\frac{da}{db}I_n
     =\frac abI_n.
 \end{aligned}
\]
The same witnesses are admissible over $M_n(\Q)$, so both image sets
are exactly $\Q I_n$.
\end{proof}

\begin{corollary}\label{cor:rational-chosen-ranges}
The polynomials $Q_n$ constructed in
Proposition~\ref{prop:exact-range}, and the explicit alternatives in
Section~\ref{sec:dimension-two} and Appendix~\ref{app:parameters}, satisfy
\begin{equation}\label{eq:rational-chosen-ranges}
 Q_n\bigl(M_n(\Q_{\geq0})^{s_n}\bigr)
 =Q_n\bigl(M_n(\Q)^{s_n}\bigr)=\Q I_n.
\end{equation}
\end{corollary}

\begin{proof}
In the B\'ezout construction, every monomial in each copy of $c$ has
degree $r$.  Both $G$ and $Q_n=G(\mathbf U)-G(\mathbf V)$ are therefore
homogeneous of degree $r$.  The explicit polynomial $U-V$ has degree
$1$; $C_2$ and $B_2$ have degree $4$; and $\widetilde Q_n$ has degree
$2n^3+n+1$.  The uncompressed Regev polynomial
$(-1)^{n-1}R_n$, when used as $Q_n$, has degree $2n^2$.
Their exact integral ranges were established in
Propositions~\ref{prop:exact-range}, \ref{prop:two},
\ref{prop:two-variables}, and \ref{prop:three-variables}, and in
\eqref{eq:regev-range}.  Lemma~\ref{lem:rational-homogeneous-range}
therefore applies to each choice.

For the B\'ezout choice, the two blocks also give direct witnesses for
every rational scalar.  Use Definition~\ref{def:bezout-polynomials}, so that
$\sum_j\lambda_j\beta_j=d=d_n$ and
$Q_n=G(\mathbf U)-G(\mathbf V)$.
Given $q=a/b$, with $a\in\Z$ and $b\geq1$, assign
\[
 \begin{aligned}
 U_{j,1}&=\frac{\max(a,0)}{bd}B_{j,1},
 &U_{j,\nu}&=B_{j,\nu}&& (2\leq\nu\leq r),\\
 V_{j,1}&=\frac{\max(-a,0)}{bd}B_{j,1},
 &V_{j,\nu}&=B_{j,\nu}&& (2\leq\nu\leq r),
 \end{aligned}
 \qquad 1\leq j\leq N_c.
\]
All these matrices have nonnegative rational entries.  Linearity in
the first argument of each copy of $c$, followed by the B\'ezout
identity, gives
\[
 \begin{aligned}
 G(\mathbf U)
   &=\frac{\max(a,0)}{bd}
       \left(\sum_j\lambda_j\beta_j\right)I_n
     =\frac{\max(a,0)}{b}I_n,\\
 G(\mathbf V)
   &=\frac{\max(-a,0)}{bd}
       \left(\sum_j\lambda_j\beta_j\right)I_n
     =\frac{\max(-a,0)}{b}I_n.
 \end{aligned}
\]
Subtracting yields $Q_n(\mathbf U,\mathbf V)=(a/b)I_n$.
This direct construction scales the first slot in each copy of $c$;
the general homogeneous argument above scales every argument of the
whole polynomial simultaneously.
\end{proof}

For example, the two-variable construction has the direct rational
parametrization
\[
 X_0=\begin{pmatrix}1&1\\0&0\end{pmatrix},\qquad
 Y(u,v)=\begin{pmatrix}u&v\\1&0\end{pmatrix},\qquad
 B_2(X_0,Y(u,v))=(1+u-v)I_2.
\]
For any $q\in\Q$, the choices $u=\max(q-1,0)$ and
$v=\max(1-q,0)$ are nonnegative rationals and give $qI_2$.
Centrality alone does not imply this surjectivity: the central polynomial
$[X,Y]^4=B_2(X,Y)^2$ has image
\[
 \{q^2I_2:q\in\Q\}\subsetneq\Q I_2
\]
on rational matrices.

\begin{remark}\label{rem:rational-nonnegative-range}
For the polynomial $H_n=\sum_{j=1}^4(Q_n^{(j)})^2$ of
Corollary~\ref{cor:natural-range}, the rational range statement is
\[
 H_n\bigl(M_n(\Q_{\geq0})^{4s_n}\bigr)
 =H_n\bigl(M_n(\Q)^{4s_n}\bigr)=\Q_{\geq0}I_n.
\]
Every value is a sum of four rational scalar squares.  Conversely, write
a nonnegative rational as $a/b$, with $a\in\N$ and $b\geq1$.
By \eqref{eq:four-squares}, choose integers $z_1,\ldots,z_4$ with
$ab=\sum_jz_j^2$.  Then $a/b=\sum_j(z_j/b)^2$, and
Corollary~\ref{cor:rational-chosen-ranges} realizes the four scalar values
independently.  Thus the integral range $d_n^2\N I_n$ in
Corollary~\ref{cor:natural-range} is replaced by $\Q_{\geq0}I_n$ when
rational arguments are allowed.
\end{remark}

\subsection{An unconditional equivalence of decision problems}

\begin{theorem}\label{thm:rational-equivalence}
For every fixed $n\geq1$,
\begin{equation}\label{eq:rational-many-one}
 \Eq_n(\Q)\equiv_m\Sys_n(\Q)\equiv_m\Hten(\Q).
\end{equation}
The reductions are uniform in $n$.  More precisely, for
$f\in\Z[x_1,\ldots,x_k]$ and $\ell=ks_n$, the same polynomial
$F_{n,f}$ constructed in Theorem~\ref{thm:main-reduction} also satisfies
\begin{equation}\label{eq:rational-main-equivalence}
\begin{aligned}
 (\exists\mathbf a\in\Q^k)\ f(\mathbf a)=0
 &\iff (\exists\mathbf A\in M_n(\Q_{\geq0})^\ell)\ F_{n,f}(\mathbf A)=0\\
 &\iff (\exists\mathbf A\in M_n(\Q)^\ell)\ F_{n,f}(\mathbf A)=0.
\end{aligned}
\end{equation}
\end{theorem}

\begin{proof}
Use the B\'ezout construction or any explicit choice covered by
Corollary~\ref{cor:rational-chosen-ranges}.  At a rational assignment,
write $Q_n^{(i)}=d_na_iI_n$ with $a_i\in\Q$.  The calculation in
\eqref{eq:evaluation-identity} remains valid and gives
\[
 F_{n,f}=d_n^D f(a_1,\ldots,a_k)I_n.
\]
Since $d_n\ne0$, a rational matrix zero yields a rational zero of $f$.
Conversely, every required tuple $d_na_iI_n$ is realizable by the copies
$Q_n^{(i)}$ on nonnegative rational arguments.  Their variable sets are
pairwise disjoint, so these values can be prescribed independently.
This proves \eqref{eq:rational-main-equivalence} and the reduction
$\Hten(\Q)\leq_m\Eq_n(\Q)$.

In the other direction, expand a finite matrix system entrywise, as in
Lemma~\ref{lem:entrywise}, now using rational entry variables.  If the
resulting scalar residuals are $p_1,\ldots,p_N$, the system is equivalent
over $\Q$ to the single equation
\[
 p_1^2+\cdots+p_N^2=0.
\]
Indeed, rational squares are nonnegative, so their sum is zero exactly
when every residual is zero.  This proves
$\Sys_n(\Q)\leq_m\Hten(\Q)$.  A single equation is a one-equation
system.  All constructions are effective uniformly in $n$.
\end{proof}

These problems are computably enumerable (c.e.): enumerate rational tuples and
evaluate the input polynomials exactly.  The theorem also gives a
solvability-preserving transformation from rational matrix systems to one
rational matrix equation, with the same scope qualification as
Remark~\ref{rem:scope-conjunction}.

The rational equivalences in \eqref{eq:rational-main-equivalence} form a
separate row from the integer equivalences in
\eqref{eq:main-equivalence}.  They cannot be joined into one chain.
For example, $f(x)=2x-1$ has a rational zero and no integer zero; its
constructed matrix equation is solvable over rational matrices but not
over integer or nonnegative integer matrices.

\subsection{Recent arithmetic claims and their consequences}

The range statements and reductions above are independent of any
undecidability theorem over $\Q$.  We now state three recent claims
concerning rational Diophantine equations and derive further
consequences from them.  Their arithmetic proofs have not been
independently verified here, and all conclusions that use these claims
are stated conditionally.

\emph{Rational undecidability.} OpenAI
\cite[Theorem~1.1 and Corollary~7.1]{OpenAI2026Q} reports
\begin{equation}\label{eq:openai-undecidability}
 \Hten(\Q)\text{ is undecidable, even for total degree at most }4.
\end{equation}
\emph{A bound on the number of variables.}
Sun \cite[Theorem~1.1]{Sun2026Q} reports
\begin{equation}\label{eq:sun-variable-bound}
 \Hten(\Q)\text{ is undecidable with at most }86\text{ variables}.
\end{equation}
The number of variables in \eqref{eq:openai-undecidability} is unrestricted;
the assertion \eqref{eq:sun-variable-bound} supplies no degree bound here.

\begin{definition}
A subset $A\subseteq\Q^a$ is \emph{Diophantine over $\Q$} if there
exist an integer $b\geq0$ and a polynomial
$D\in\Z[T_1,\ldots,T_a,W_1,\ldots,W_b]$ such that
\[
 \mathbf t\in A
 \quad\Longleftrightarrow\quad
 (\exists\mathbf w\in\Q^b)\ D(\mathbf t,\mathbf w)=0.
\]
Thus the definition uses only existentially quantified polynomial
equations, with the displayed tuple $\mathbf t$ left free.
\end{definition}

\emph{A Diophantine definition of the integers.}
The project coordinated by Elliot Glazer, in Version~16 of
\cite[Theorem~1.1]{GlazerProject2026Q}, reports an explicit
quartic polynomial $D\in\Z[T,W_1,\ldots,W_{313}]$ for which
\begin{equation}\label{eq:integer-definition}
 t\in\Z
 \quad\Longleftrightarrow\quad
 (\exists\mathbf w\in\Q^{313})\ D(t,\mathbf w)=0
 \qquad(t\in\Q).
\end{equation}
This is an existential Diophantine definition of $\Z$ in $\Q$.
A first-order definition involving universal quantifiers would not
suffice for the following reduction.

\begin{proposition}\label{prop:rational-integer-transfer}
Assuming \eqref{eq:integer-definition}, one has
$\Hten(\Z)\leq_m\Hten(\Q)$.
\end{proposition}

\begin{proof}
Given $f\in\Z[x_1,\ldots,x_k]$, introduce $k$ blocks of fresh variables
$\mathbf w^{(i)}=(w_{i,1},\ldots,w_{i,313})$.
Their variable sets are pairwise disjoint and disjoint from the variables
$x_1,\ldots,x_k$.  Form
\begin{equation}\label{eq:rational-integer-transfer}
 h(\mathbf x,\mathbf w^{(1)},\ldots,\mathbf w^{(k)})
 =f(\mathbf x)^2+\sum_{i=1}^kD(x_i,\mathbf w^{(i)})^2.
\end{equation}
At a rational zero of $h$, every squared residual is zero.  The
definition \eqref{eq:integer-definition} forces each $x_i$ to be an integer,
and $f(\mathbf x)=0$.  Conversely, an integer zero of $f$ extends to a
rational zero of $h$ by choosing witnesses for the $k$ occurrences of
$D$ independently.  The polynomial $D$ is fixed, so this is an effective
transformation of input polynomials.
\end{proof}

\begin{corollary}\label{cor:rational-hardness}
Assuming \eqref{eq:openai-undecidability}, the problems $\Eq_n(\Q)$ and
$\Sys_n(\Q)$ are undecidable for every fixed $n\geq1$.
Assuming \eqref{eq:integer-definition}, both problems are computably enumerable
complete under computable many-one reductions.
\end{corollary}

\begin{proof}
The first assertion follows from Theorem~\ref{thm:rational-equivalence}.
For the second, Proposition~\ref{prop:rational-integer-transfer} and
the c.e. completeness of $\Hten(\Z)$ give c.e. hardness of
$\Hten(\Q)$.  Its computable enumerability gives completeness, which
then transfers by \eqref{eq:rational-many-one}.  Undecidability alone
would not justify the completeness assertion.
\end{proof}

\begin{corollary}\label{cor:rational-larger-languages}
Assuming \eqref{eq:openai-undecidability}, for every fixed $n\geq1$,
single-equation solvability over $M_n(\Q)$ remains undecidable when
additional matrix parameters are permitted.  Its existential theory and
its $\exists\forall$ theory are also undecidable.
\end{corollary}

\begin{proof}
As in Corollary~\ref{cor:larger-languages}, each expanded problem
contains the scalar-coefficient single-equation problem as a subclass.
For the $\exists\forall$ prefix, append a universally quantified
variable that does not occur in the equation.
\end{proof}

\subsection{Shared witnesses and separate quantitative bounds}

\begin{proposition}\label{prop:rational-shared-witnesses}
Given a scalar polynomial $f$ in $k$ variables, a single matrix equation
equivalent to its rational solvability can be constructed using $k$
matrix variables for $n=1$, $k+1$ for $n=2$, and $k+2$ for $n\geq3$.
For $n\geq2$, every scalar rational zero also gives a solution with
entrywise nonnegative rational matrix arguments.
\end{proposition}

\begin{proof}
For $n=1$, interpret $f$ directly as a polynomial on $1\times1$ matrices,
writing each commutative monomial in a fixed variable order.

For $n=2$, substitute
$B_2(X,Y_1),\ldots,B_2(X,Y_k)$ into $f$, with one shared variable
$X$.  Every substituted value is a rational scalar.  The common choice
$X=X_0$ and independent choices of the $Y_i$ realize any scalar tuple,
by the parametrization above, using nonnegative rational matrices.

For $n\geq3$, substitute
$\widetilde Q_n(U_1,S,T),\ldots,\widetilde Q_n(U_k,S,T)$, sharing
$S,T$.  Every value is a rational scalar, and
\eqref{eq:compressed-witness-value} realizes a prescribed tuple
$(a_1I_n,\ldots,a_kI_n)$ by setting $S=S_0$, $T=T_0$, and
\[
 U_i=\operatorname{diag}\left(
 \frac{\max(a_i,0)}{\delta_n},
 \frac{\max(-a_i,0)}{\delta_n},0,\ldots,0\right).
\]
The $U_i$ are chosen independently and are nonnegative rational matrices.
All substitutions are integer-coefficient polynomial expressions; the
displayed matrices are witnesses, not parameters in the output equation.
\end{proof}

\begin{corollary}\label{cor:rational-variable-bounds}
Assuming \eqref{eq:sun-variable-bound}, single-equation solvability over $M_n(\Q)$
is undecidable even when the number of matrix variables is at most
\begin{equation}\label{eq:rational-variable-bounds}
 v_n=\begin{cases}
 86,&n=1,\\
 87,&n=2,\\
 88,&n\geq3.
 \end{cases}
\end{equation}
\end{corollary}

\begin{proof}
Apply Proposition~\ref{prop:rational-shared-witnesses} to the class of
scalar inputs in \eqref{eq:sun-variable-bound}.  In dimension one no auxiliary
parametrization is needed.
\end{proof}

For degree bounds, we also record an elementary normalization whose
proof is independent of the three recent arithmetic claims.

\begin{lemma}[Quartic normalization]\label{lem:rational-quartic-normalization}
Solvability of arbitrary scalar polynomial equations over $\Q$ reduces
effectively to solvability of equations of total degree at most $4$.
This normalization can increase the number of variables.
\end{lemma}

\begin{proof}
Write an arithmetic computation of the input polynomial using additions,
multiplications, and integer constants.  Introduce a fresh variable for
each intermediate value.  An addition is enforced by $z-u-v=0$, a
multiplication by $z-uv=0$, and a constant assignment by $z-a=0$.
Append the equation saying that the final computed value is zero.
These equations have degree at most $2$, and their simultaneous solutions
are exactly the input assignments extended by their computed intermediate
values.  Replace the equations by the sum of the squares of their
residuals.  Over $\Q$ this preserves solvability and gives one polynomial
of total degree at most $4$.  Constant inputs can be treated directly.
\end{proof}

\begin{corollary}\label{cor:rational-degree-bounds}
Assuming \eqref{eq:openai-undecidability}, single-equation solvability over
$M_n(\Q)$ is undecidable even for total degree at most
\begin{equation}\label{eq:rational-degree-bounds}
 e_n=\begin{cases}
 4,&n=1,\\
 16,&n=2,\\
 4(2n^3+n+1),&n\geq3.
 \end{cases}
\end{equation}
Assuming \eqref{eq:integer-definition}, each of these fixed-degree problems
is computably enumerable complete.  In this corollary the number of
variables is unrestricted.
\end{corollary}

\begin{proof}
A scalar quartic is unchanged in dimension one.  Substitution of $B_2$,
of degree $4$, gives degree at most $16$.  Substitution of
$\widetilde Q_n$, of degree $2n^3+n+1$, gives degree at most
$4(2n^3+n+1)$.  The shared witnesses used in
Proposition~\ref{prop:rational-shared-witnesses} prove the required
solvability equivalences.  This establishes undecidability from
\eqref{eq:openai-undecidability}.

Under \eqref{eq:integer-definition}, Proposition~\ref{prop:rational-integer-transfer}
first gives a reduction from $\Hten(\Z)$ to rational equations.
Lemma~\ref{lem:rational-quartic-normalization} makes these equations
quartic, and the preceding matrix substitutions give the stated degree
bounds.  This proves c.e. hardness; enumeration of rational matrix tuples
proves membership in the c.e. class.
\end{proof}

The bounds \eqref{eq:rational-variable-bounds} and
\eqref{eq:rational-degree-bounds} concern separate restrictions on the
input problems.  No simultaneous bound on both the number of variables
and the degree is asserted.  Nor does undecidability in the bounded-variable
class alone establish c.e. completeness of that bounded-variable class.

\section{Restricted entries and matrix subsemirings}
\label{sec:restricted}

Undecidability over a full matrix domain does not by itself imply
undecidability over an arbitrary subset.  The following two transfer
arguments provide such conclusions for particular entry restrictions and
particular subsemirings.

\subsection{Entries in an additive basis}

\begin{definition}
A set $B\subseteq\N$ is an \emph{additive basis of finite order} if there
is a positive integer $h$ such that every $a\in\N$ can be written
$a=b_1+\cdots+b_h$ with $b_j\in B$.  Necessarily $0,1\in B$.
For such a set, $M_n(B)$ denotes the set of matrices whose entries belong to
$B$.  Polynomial equations on this set are evaluated in $M_n(\Z)$; no
closure of $B$ or $M_n(B)$ under the operations is assumed.
\end{definition}

\begin{theorem}[Additive-basis transfer]\label{thm:additive-basis}
Let $B\subseteq\N$ be an additive basis of order $h$, and fix $n\geq1$.
Solvability of a single scalar-coefficient polynomial equation over
$M_n(\N)$ reduces effectively to solvability of a single such equation
over $M_n(B)$.  Consequently the latter problem is undecidable.  If $B$ is
computably enumerable, it is computably enumerable complete.
\end{theorem}

\begin{proof}
For each variable $X_i$ of an input polynomial $P$, introduce a block of
$h$ new variables.  Require the blocks for different $i$ to have pairwise
disjoint variable sets, and substitute
\begin{equation}\label{eq:basis-substitution}
 X_i=X_{i,1}+\cdots+X_{i,h}.
\end{equation}
Every sum on the right is a nonnegative matrix.  Conversely, every
nonnegative matrix is such a sum: choose an $h$-term representation in $B$
independently for each entry and group the corresponding summands into
$h$ matrices.  Thus the single equation obtained by substituting
\eqref{eq:basis-substitution} into $P=0$ is solvable over $M_n(B)$ exactly
when $P=0$ is solvable over $M_n(\N)$.

The transformation uses only the fixed integer $h$; it does not have to
compute representations of entries.  Undecidability follows from
Corollary~\ref{cor:undecidability}.  If $B$ is computably enumerable, enumerate
tuples of matrices with entries in $B$ and evaluate the equation.  This
semidecides solvability, while the preceding reduction gives computably
enumerable hardness.
\end{proof}

\begin{corollary}\label{cor:basis-examples}
For every fixed $n\geq1$, single-equation solvability is undecidable, and
indeed computably enumerable complete, when every matrix entry is required
to belong to any one of the following fixed sets:
\begin{enumerate}
\item the nonnegative $t$th powers, for a fixed $t\geq2$;
\item the polygonal numbers of a fixed order $s\geq3$, including zero;
\item the prime numbers together with $0$ and $1$.
\end{enumerate}
\end{corollary}

\begin{proof}
Hilbert's solution of Waring's problem says that the nonnegative $t$th
powers form an additive basis of finite order \cite{Hilbert1909}.
The polygonal numbers of order $s$ are
\[
 \left\{\frac{(s-2)a^2-(s-4)a}{2}:a\in\N\right\}.
\]
Cauchy's polygonal-number theorem represents every nonnegative integer
as a sum of $s$ such numbers; see \cite{Nathanson1987}.
Schnirelmann's theorem gives a finite bound for the number of prime summands
needed to represent every integer at least two \cite{Schnirelmann1933};
adjoining $0$ and $1$ and padding with zeros therefore gives an additive
basis of all of $\N$.  All these sets are computably enumerable, so
Theorem~\ref{thm:additive-basis} applies.
\end{proof}

Scalar Diophantine problems with related restrictions on their unknowns
are treated in \cite{Prunescu2006}.  Here the restriction is imposed on
each matrix entry, and the substitution preserves the use of a single
equation.

\subsection{Homomorphisms and scalar retracts}

\begin{proposition}[Homomorphism transfer]\label{prop:homomorphisms}
Suppose that unital rings or semirings $A$ and $B$ admit homomorphisms
$\alpha:A\to B$ and $\beta:B\to A$ preserving $0$ and $1$.  Then they
satisfy exactly the same scalar-coefficient equations in existential form,
and exactly the same existential finite systems of such equations.
The same conclusion holds for their positive existential theories in the
signature $\{+,\cdot,0,1\}$.
\end{proposition}

\begin{proof}
Write equations as equalities between terms with nonnegative coefficients.
A homomorphism preserves their evaluations.  Applying $\alpha$ to any
solution gives a solution in $B$, and applying $\beta$ gives the converse.
The same tuple map preserves simultaneous equations, and homomorphisms
also preserve finite disjunctions and existential quantifiers.  No
effectivity assumption on the homomorphisms is needed: the identical input
sentence is true on both sides.
\end{proof}

A pair of homomorphisms as above with $\beta\circ\alpha$ equal to the
identity on $A$ makes $A$ a \emph{retract} of $B$.  The following examples
have this stronger property, with the scalar domain as the retract.
They cover several natural subrings and subsemirings.
A \emph{preorder} $\preccurlyeq$ on $\{1,\ldots,n\}$ is a reflexive and
transitive relation.  Put
\begin{equation}\label{eq:preorder-matrices}
 \mathcal B_{\preccurlyeq}(R)
 =\{A\in M_n(R): A_{ab}=0\text{ whenever }a\not\preccurlyeq b\}.
\end{equation}
The relation $a\sim b$ defined by $a\preccurlyeq b$ and
$b\preccurlyeq a$ is an equivalence relation.

\begin{proposition}\label{prop:preorder-retract}
Let $R\in\{\N,\Z,\Q\}$.  If the $\sim$-class of some index $i$ is
the singleton $\{i\}$, then $\mathcal B_{\preccurlyeq}(R)$ is a unital
subsemiring of $M_n(R)$ and has exactly the same scalar-coefficient
Diophantine theory as $R$.  In particular, for $R=\N$ or $R=\Z$,
single-equation solvability in this subsemiring is undecidable.
For $R=\Q$, its single-equation and finite-system problems are many-one
equivalent to $\Hten(\Q)$.  They are therefore undecidable under
\eqref{eq:openai-undecidability}, and computably enumerable complete under
\eqref{eq:integer-definition}.
\end{proposition}

\begin{proof}
Reflexivity ensures $I_n\in\mathcal B_{\preccurlyeq}(R)$, and closure
under addition is immediate.  If $a\not\preccurlyeq b$, transitivity
implies that, for every $c$, at least one of $a\preccurlyeq c$ and
$c\preccurlyeq b$ fails.  Hence every summand $A_{ac}B_{cb}$ is zero,
which proves closure under multiplication.  For rings the set is also
closed under additive inverses.

The map $\alpha:R\to\mathcal B_{\preccurlyeq}(R)$ given by
$\alpha(a)=aI_n$ is a unital homomorphism.  Define $\beta(A)=A_{ii}$.
For $j\ne i$ the singleton assumption implies that at least one of
$i\preccurlyeq j$ and $j\preccurlyeq i$ fails.  Consequently
\[
 (AB)_{ii}=\sum_j A_{ij}B_{ji}=A_{ii}B_{ii}.
\]
Thus $\beta$ is a unital homomorphism, and $\beta\circ\alpha$ is the
identity of $R$.  Proposition~\ref{prop:homomorphisms} gives the asserted
equality of theories.  For $R\in\{\N,\Z\}$, scalar undecidability
gives the corresponding assertion.  Over $\Q$, scalar systems reduce to
one equation by summing their squared residuals.  Equality of theories is
literal on the common input syntax, so it gives the stated many-one
equivalences.  The remaining conclusions follow from
Corollary~\ref{cor:rational-hardness} with $n=1$.
\end{proof}

Examples include diagonal matrices and upper triangular matrices.  For a
fixed index $i$, they also include the following five families; all entries
not required to be zero are arbitrary:
\begin{enumerate}
\item $\Sigma_i(R)$: the off-diagonal entries in both row $i$ and column
$i$ are zero;
\item $\Gamma_i(R)$: the off-diagonal entries in column $i$ are zero;
\item $\Lambda_i(R)$: the off-diagonal entries in row $i$ are zero;
\item $\mathcal R_i(R)$: the entries in
$\{i,\ldots,n\}\times\{1,\ldots,i\}$ other than $(i,i)$ are zero;
\item $\mathcal{RR}_i(R)$: the entries in the union
$\bigl(\{i,\ldots,n\}\times\{1,\ldots,i\}\bigr)
\cup\bigl(\{1,\ldots,i\}\times\{i,\ldots,n\}\bigr)$ other than
$(i,i)$ are zero.
\end{enumerate}
For $\Sigma_i$ the two blocks $\{i\}$ and its complement are diagonal
blocks.  For $\Gamma_i$, use the ordered blocks
$\{i\},\{1,\ldots,n\}\setminus\{i\}$; for $\Lambda_i$, use the
reverse order.  In each of these two orders the pattern is block upper
triangular.  For $\mathcal R_i$ the ordered
blocks are $\{1,\ldots,i-1\}$, $\{i\}$, and
$\{i+1,\ldots,n\}$; for $\mathcal{RR}_i$ the same three blocks are
diagonal blocks.  Empty blocks are omitted.  These descriptions verify
the preorder hypothesis and display the singleton block explicitly.

\section{An elementary definition of diagonal elementary matrices}
\label{sec:corners}

The next result concerns an explicitly definable set, rather than only a
reduction of decision problems.  It also provides a short elementary
reduction to systems over nonnegative matrices, independent of the
central-polynomial theorem.

\begin{lemma}[Selecting a diagonal matrix unit]\label{lem:selector}
In $M_n(\N)$ the formula
\begin{equation}\label{eq:selector}
\begin{split}
 \operatorname{Sel}_n(Y):\quad
 \exists A_1,\ldots,A_{n-1}\quad
 &Y+\sum_{j=1}^{n-1}A_jYA_j=I_n,\\
 &A_1^2A_2^2\cdots A_{n-1}^2=I_n
\end{split}
\end{equation}
defines precisely the set $\{E_{11},\ldots,E_{nn}\}$.
The sum is empty and the product equals $I_1$ when $n=1$.
\end{lemma}

\begin{proof}
The case $n=1$ is immediate.  For $n\geq2$, the second equation implies
that every $A_j$ is invertible over $\Q$.  The first equation implies
$Y\ne0$, so all $n$ summands $Y,A_1YA_1,\ldots,A_{n-1}YA_{n-1}$ are
nonzero nonnegative integer matrices.  Their sum is $I_n$.  Since the sum
of all entries of $I_n$ is $n$, each summand has exactly one entry equal
to $1$ and all its other entries equal to zero.  The nonzero entry must
be diagonal.  Thus the summands are the $n$ distinct diagonal matrix units,
and in particular $Y=E_{ii}$ for some $i$.

Conversely, fix $Y=E_{ii}$ and choose the $A_j$ to be the permutation
matrices of the $n-1$ transpositions interchanging $i$ with each of the
other indices.  They have square $I_n$, and their conjugations send
$E_{ii}$ to the remaining diagonal elementary matrices.  Both equations follow.
\end{proof}

For a selected $Y=E_{ii}$, the single equation $YZY=Z$ defines the corner
\[
 \{Z\in M_n(\N):YZY=Z\}=\{aY:a\in\N\}.
\]
With identity $Y$, this corner is isomorphic to the scalar semiring $\N$.
The formula selects a diagonal elementary matrix without naming one in the
language.  Its proof uses nonnegativity essentially, and does not assert
the same definition over $M_n(\Z)$.

\begin{proposition}[An elementary three-equation reduction]
\label{prop:three-equations}
Given $f\in\Z[x_1,\ldots,x_k]$, one can effectively construct a system
of three scalar-coefficient equations in $n+k$ matrix variables that is
solvable over $M_n(\N)$ if and only if $f=0$ has a solution over $\N$.
\end{proposition}

\begin{proof}
Order the factors in each scalar monomial once and for all.  Replace a
nonconstant monomial $x_{i_1}\cdots x_{i_t}$ by
$YX_{i_1}Y\cdots YX_{i_t}Y$, and replace the empty monomial $1$ by $Y$.
Extend this replacement linearly to obtain $\widehat f(Y,\mathbf X)$.
The required system consists of the two equations in
\eqref{eq:selector}, with their witnesses $A_j$ now treated as unknowns,
and
\begin{equation}\label{eq:corner-equation}
 \widehat f(Y,X_1,\ldots,X_k)=0.
\end{equation}
If $Y=E_{ii}$, then $YX_jY=(X_j)_{ii}Y$ and therefore
\[
 \widehat f(Y,\mathbf X)
 =f\bigl((X_1)_{ii},\ldots,(X_k)_{ii}\bigr)Y.
\]
Because $Y\ne0$, any solution of the system yields a nonnegative scalar
zero of $f$.  Conversely, from a scalar zero $(a_1,\ldots,a_k)$, choose
$Y=E_{11}$, witnesses from Lemma~\ref{lem:selector}, and $X_j=a_jY$.
This gives a solution.  The variables are $Y$, the $n-1$ matrices $A_j$,
and the $k$ matrices $X_j$.
\end{proof}

The construction preserves the concrete corner interpretation.  Its
undecidability consequence is already contained in the stronger
single-equation theorem; the explicit definition and the bound of three
equations are separate information.

\section{Constant-free and homogeneous equations}
\label{sec:nontrivial}

An equation with no constant term always has the zero solution.  In this
section solvability means the existence of a \emph{nontrivial} tuple,
that is, a tuple in which at least one matrix is nonzero.  All inputs have
at least one variable, and all coefficients remain scalar integers.
The declared variable list $X_1,\ldots,X_k$, including any unused
variables, is part of the input.

\begin{definition}
For $R\in\{\N,\Z,\Q\}$ let $\FEq_n(R)$ be the set of polynomials
$P\in\free{X_1,\ldots,X_k}$ with $P(0,\ldots,0)=0$ that have a
nontrivial zero in $M_n(R)^k$.  Let $\FSys_n(R)$ be the analogous set
of finite systems whose polynomials all have zero constant term.
Let $\HEq_n(R)$ and $\HSys_n(R)$ denote the corresponding problems
restricted to homogeneous polynomials of positive degree.  Zero
polynomials are also allowed and impose no condition; they are discarded
when selecting equations of a given degree.  Different nonzero equations
in a homogeneous system may have different degrees.
\end{definition}

If homogeneous polynomials of degree zero are allowed as additional
inputs, first reject any system containing a nonzero constant equation.
After discarding zero equations, the arguments below apply to all
remaining inputs.

\subsection{A strict chain for constant-free equations}

\begin{lemma}\label{lem:zero-padding}
If $n\leq m$, zero padding $A\mapsto\operatorname{diag}(A,0_{m-n})$
preserves solutions of constant-free equations and systems, as well as
nontriviality.  It also preserves homogeneity of the input equations.
\end{lemma}

\begin{proof}
The map is injective and preserves addition, multiplication, and zero.
Every monomial of positive degree evaluates in the upper left block,
with zero in the other block.  Thus for every polynomial of zero constant
term,
\[
 P\bigl(\operatorname{diag}(A_1,0),\ldots,
        \operatorname{diag}(A_k,0)\bigr)
 =\operatorname{diag}\bigl(P(A_1,\ldots,A_k),0\bigr).
\]
The map does not preserve $1$ when $n<m$, which explains the restriction
on constant terms.
\end{proof}

\begin{lemma}[A threshold dimension]\label{lem:threshold}
For positive integers $r,m$ and $R\in\{\N,\Z,\Q\}$, the equation
\begin{equation}\label{eq:threshold}
 X^{r+1}-2X=0
\end{equation}
with declared variable list $(X)$ has a nonzero solution in $M_m(R)$
if and only if $m\geq r$.
\end{lemma}

\begin{proof}
The matrix $T_r$ from Lemma~\ref{lem:dimension-detector} is nonzero and
satisfies \eqref{eq:threshold}.  Zero padding gives witnesses in every
dimension at least $r$.

Conversely, let $0\ne A\in M_m(\Q)$ satisfy \eqref{eq:threshold}.
Its minimal polynomial $\mu_A$ divides $T(T^r-2)$.
The factors $T$ and $T^r-2$ are relatively prime, and the latter is
irreducible by Eisenstein's criterion.  If $\mu_A$ did not contain the
factor $T^r-2$, it would divide $T$, forcing $A=0$.  Hence
$\deg\mu_A\geq r$.  By the Cayley--Hamilton theorem
$\deg\mu_A\leq m$, so $m\geq r$.
\end{proof}

\begin{theorem}[Ordinary dimension order]\label{thm:constant-free-order}
For $R\in\{\N,\Z,\Q\}$ and $n,m\geq1$,
\begin{equation}\label{eq:constant-free-order}
 \FEq_n(R)\subseteq\FEq_m(R)
 \quad\Longleftrightarrow\quad
 \FSys_n(R)\subseteq\FSys_m(R)
 \quad\Longleftrightarrow\quad n\leq m.
\end{equation}
In particular, both families form strictly increasing chains in the
ordinary order of the positive integers.
\end{theorem}

\begin{proof}
The forward-in-dimension inclusions follow from Lemma~\ref{lem:zero-padding}.
If $n>m$, the polynomial $X^{n+1}-2X$ belongs to $\FEq_n(R)$ and not
to $\FEq_m(R)$, by Lemma~\ref{lem:threshold}; the same equation as a
one-equation system distinguishes the system problems.  If $n<m$, use
$X^{m+1}-2X$ to prove strictness.
\end{proof}

\subsection{Homogeneous equations in dimension one}
\label{subsec:scalar-homogeneous}

In dimension one, nontrivial homogeneous solvability has the same
computable many-one degree as Hilbert's tenth problem over $\Q$.
The relation between homogeneous integer equations and rational
solvability is classical; the Pell-equation construction used below is
given in \cite[Section~7.4, pp.~146--149]{Matiyasevich1993}.
We give the argument explicitly and introduce a common nonnegative
offset, obtaining one output equation that works over $\N$, $\Z$, and
$\Q$ simultaneously.

Scalar polynomials in this subsection are written in commutative
notation.  Collecting words, or fixing a variable order in each
commutative monomial, converts effectively between this notation and the
input syntax for dimension one.  Declared variables are retained even
when they do not occur in a polynomial.

\begin{lemma}[An explicit homogeneous reduction]
\label{lem:scalar-homogeneous-octic}
Given $g\in\Z[x_1,\ldots,x_m]$ of total degree at most $4$, one can
effectively construct a nonzero homogeneous polynomial $P_g$ of degree
$8$, in $m+10$ variables, such that
\begin{equation}\label{eq:scalar-homogeneous-octic-equivalence}
\begin{aligned}
 g=0\text{ has a rational solution}
 &\iff P_g=0\text{ has a nontrivial solution over }\N\\
 &\iff P_g=0\text{ has a nontrivial solution over }\Z\\
 &\iff P_g=0\text{ has a nontrivial solution over }\Q.
\end{aligned}
\end{equation}
The construction also applies to constant and zero polynomials $g$.
\end{lemma}

\begin{proof}
Write $g=\sum_{|\alpha|\leq4}c_\alpha\mathbf x^\alpha$, where
$\alpha\in\N^m$, $|\alpha|=\sum_i\alpha_i$, and
$\mathbf x^\alpha=\prod_i x_i^{\alpha_i}$.  Introduce the distinct variables
$u_1,\ldots,u_m,v,q,s_1,\ldots,s_4,t_1,\ldots,t_4$ and define
\begin{align}
 H(\mathbf u,v,q)
 &=\sum_{|\alpha|\leq4}c_\alpha q^{4-|\alpha|}
       \prod_{i=1}^m(u_i-v)^{\alpha_i},
 \label{eq:scalar-homogeneous-H}\\
 A&=\sum_{i=1}^m u_i^2+v^2+\sum_{j=1}^4s_j^2,
 \qquad B=\sum_{j=1}^4t_j^2,
 \label{eq:scalar-homogeneous-AB}\\
 \Phi&=A^2-4AB+B^2-q^4,
 \qquad P_g=H^2+\Phi^2.
 \label{eq:scalar-homogeneous-P}
\end{align}
The declared variables of $P_g$ are exactly the $m+10$ symbols just
listed; $A$, $B$, $H$, and $\Phi$ are polynomial abbreviations.
The expanded formula \eqref{eq:scalar-homogeneous-H} is defined also
at $q=0$.  Both $H$ and $\Phi$ are homogeneous quartics, with the zero
polynomial allowed for $H$, so $P_g$ is homogeneous of degree $8$.
It is nonzero: at $s_1=1$ and all other variables zero, its value is $1$.

Suppose that $P_g=0$ at a nontrivial rational tuple.  Both squared
residuals vanish.  If $q=0$, then $\Phi=0$ gives
\[
 (A-2B)^2=3B^2.
\]
For rational $A,B$, this forces $B=0$, since $3$ is not a rational
square; it then forces $A=0$.  The expressions in
\eqref{eq:scalar-homogeneous-AB} are sums of rational squares, so every
$u_i,v,s_j,t_j$ is zero.  Together with $q=0$ this contradicts
nontriviality.  Thus $q\ne0$, and
\[
 0=H(\mathbf u,v,q)
   =q^4g\left(\frac{u_1-v}{q},\ldots,
              \frac{u_m-v}{q}\right)
\]
provides a rational zero of $g$.

Conversely, let $(a_1,\ldots,a_m)$ be a rational zero of $g$.
Choose an integer $q\geq1$ such that $r_i=qa_i\in\Z$ for every $i$.
Set
\[
 v=\max\{0,-r_1,\ldots,-r_m\},\qquad u_i=r_i+v.
\]
These are nonnegative integers and make $H=0$.  Put
$C=\sum_i u_i^2+v^2$.

To choose the remaining variables, define integer sequences by
\begin{equation}\label{eq:scalar-homogeneous-pell}
 \xi_0=1,\quad \eta_0=0,\qquad
 \xi_{\nu+1}=2\xi_\nu+3\eta_\nu,\quad
 \eta_{\nu+1}=\xi_\nu+2\eta_\nu\qquad(\nu\geq0).
\end{equation}
The recurrence preserves the Pell relation because
\[
 \begin{aligned}
 \xi_{\nu+1}^2-3\eta_{\nu+1}^2
 &=(2\xi_\nu+3\eta_\nu)^2-3(\xi_\nu+2\eta_\nu)^2\\
 &=\xi_\nu^2-3\eta_\nu^2.
 \end{aligned}
\]
Thus $\xi_\nu^2-3\eta_\nu^2=1$ for every $\nu$.
All terms are nonnegative and $\xi_{\nu+1}\geq2\xi_\nu$, so
$\xi_\nu\geq2^\nu$.  Choose $\nu$ large enough that
\[
 A_0=q^2(\xi_\nu+2\eta_\nu)\geq C,
 \qquad B_0=q^2\eta_\nu\geq0.
\]
Then
\[
 A_0^2-4A_0B_0+B_0^2
 =q^4(\xi_\nu^2-3\eta_\nu^2)=q^4.
\]
By the four-square theorem \eqref{eq:four-squares}, choose
nonnegative integers $s_1,\ldots,s_4,t_1,\ldots,t_4$ satisfying
\[
 \sum_{j=1}^4s_j^2=A_0-C,
 \qquad \sum_{j=1}^4t_j^2=B_0.
\]
Integer four-square representations may be made nonnegative by taking
absolute values.  Now $A=A_0$, $B=B_0$, and $\Phi=0$, so $P_g=0$.
This tuple is nontrivial because $q\geq1$.  It lies in $\N^{m+10}$,
and hence also gives solutions over $\Z$ and $\Q$.
\end{proof}

The common offset in this construction represents all signed integers
$r_i$ as $u_i-v$ using nonnegative integers.  The norm $A$ contains
$\sum_i u_i^2+v^2$, rather than only $\sum_i(u_i-v)^2$.
Consequently $A=0$ forces every introduced coordinate to vanish, even
when several differences $u_i-v$ are zero.  This is the point that
prevents unwanted nontrivial solutions with $q=0$.

\begin{theorem}[The homogeneous scalar decision problems]
\label{thm:scalar-homogeneous-equivalence}
For each $R\in\{\N,\Z,\Q\}$,
\begin{equation}\label{eq:scalar-homogeneous-equivalence}
 \HEq_1(R)\equiv_m\HSys_1(R)\equiv_m\Hten(\Q).
\end{equation}
These equivalences are unconditional.  The reduction from
$\Hten(\Q)$ to $\HEq_1(R)$ can always output a nonzero homogeneous
polynomial of degree $8$, and the same output works for all three domains.
\end{theorem}

\begin{proof}
Given any scalar rational-solvability instance $f$, first apply
Lemma~\ref{lem:rational-quartic-normalization} to obtain an equivalent
polynomial $g$ of total degree at most $4$.  Apply
Lemma~\ref{lem:scalar-homogeneous-octic} to $g$.
This gives an explicit computable many-one reduction to each of the
three homogeneous single-equation problems.  The preliminary quartic
normalization may increase the number of variables.

For the reverse reduction over $\Z$ or $\Q$, let $p$ be a homogeneous
scalar polynomial with declared variable list $X_1,\ldots,X_k$.
Its nontrivial rational solvability is equivalent to rational solvability
of the single, not necessarily homogeneous equation
\begin{equation}\label{eq:scalar-homogeneous-upper}
 p(\mathbf X)^2+
 \left(z\sum_{i=1}^kX_i^2-1\right)^2=0.
\end{equation}
Indeed, the second residual is solvable in $z\in\Q$ exactly when
$\mathbf X\ne0$.  A nontrivial rational zero of a homogeneous
polynomial can be multiplied by a common positive denominator to give
a nontrivial integer zero.  Thus \eqref{eq:scalar-homogeneous-upper}
also decides the integer instance.  The zero polynomial causes no
exception to this argument.

For $R=\N$, first replace every declared variable $X_i$ by
$Y_{i,1}^2+\cdots+Y_{i,4}^2$, using blocks whose variable sets are
pairwise disjoint.
The resulting polynomial is homogeneous, and it has a nontrivial integer
zero exactly when the original polynomial has a nontrivial nonnegative
integer zero.  One direction uses the four-square theorem; the other
uses the fact that a sum of integer squares is zero only when all its
entries are zero.  Apply the preceding integer-to-rational reduction.
The sums here and in \eqref{eq:scalar-homogeneous-upper} include all
declared variables, so unused variables do not alter the correctness
of the nontriviality test.

Finally, a single equation is a one-equation system.  Conversely,
consider a finite homogeneous scalar system in its declared variables.
Discard its zero polynomials.  If none remain, return the zero polynomial
in the same variables.  Otherwise let $d_j\geq1$ be the degrees of the
remaining polynomials $p_j$, and let $L=\operatorname{lcm}_j d_j$
be their least common multiple.
Replace the system by
\begin{equation}\label{eq:scalar-homogeneous-system-merge}
 \sum_j p_j^{\,2L/d_j}=0.
\end{equation}
Every summand has degree $2L$ and is an even power.  Over each of
$\N,\Z,\Q$, its zero set is exactly the simultaneous zero set of the
original system, in the same variables.  In particular, nontriviality
is preserved.  This proves the system-to-equation reductions and
completes \eqref{eq:scalar-homogeneous-equivalence}.
\end{proof}

\begin{corollary}\label{cor:scalar-homogeneous-hardness}
Assuming the reported undecidability claim \eqref{eq:openai-undecidability}, each of
$\HEq_1(R)$ and $\HSys_1(R)$, for $R\in\{\N,\Z,\Q\}$, is
undecidable.  This already holds for single nonzero homogeneous
polynomials of degree $8$.
Assuming the reported Diophantine definition \eqref{eq:integer-definition}, all these problems,
including their single-equation degree-$8$ restrictions, are computably
enumerable complete under computable many-one reductions.
\end{corollary}

\begin{proof}
The claim \eqref{eq:openai-undecidability} gives rational undecidability, whereas
the definition \eqref{eq:integer-definition}, together with
Proposition~\ref{prop:rational-integer-transfer}, gives c.e. completeness
of $\Hten(\Q)$.
Apply Theorem~\ref{thm:scalar-homogeneous-equivalence} and its octic
reduction.  The homogeneous problems and their fixed-degree restrictions
are computably enumerable by enumerating nonzero tuples over the stated
domain and checking the equations exactly.
\end{proof}

The degree-eight bound is sufficient for the simultaneous construction
above.  A classical reduction of homogeneous integer equations to degree
four is recorded in \cite[Chapter~7, Exercise~7, p.~150]{Matiyasevich1993}.

Thus the scalar homogeneous problems retain the arithmetic difficulty
of rational Diophantine equations.  The higher-dimensional behavior is
different: the following theorem shows that homogeneous solvability
with a nontrivial matrix tuple is decidable in every dimension at least
two.

\subsection{Stabilization for homogeneous equations}

\begin{theorem}[Homogeneous stabilization]\label{thm:homogeneous}
Fix $R\in\{\N,\Z,\Q\}$ and $n\geq2$.
A finite system of homogeneous equations of positive degrees has a
nontrivial solution in $M_n(R)$ if and only if its degree-one equations
have a common nonzero scalar solution in $R^k$.
Consequently
\begin{equation}\label{eq:homogeneous-stabilization}
 \HEq_n(R)=\HEq_2(R),\qquad
 \HSys_n(R)=\HSys_2(R)\qquad(n\geq2),
\end{equation}
and all these decision problems are decidable.  In particular, every
homogeneous equation of degree at least two has a nontrivial matrix
solution in every dimension at least two.
\end{theorem}

\begin{proof}
Write the degree-one equations as
\[
 \sum_{i=1}^k a_{ji}X_i=0\qquad(1\leq j\leq q).
\]
If $(A_1,\ldots,A_k)$ is a nontrivial matrix solution, choose an entry
position $(u,v)$ for which some $(A_i)_{uv}\ne0$.  Then the scalar vector
$\bigl((A_1)_{uv},\ldots,(A_k)_{uv}\bigr)$ is nonzero and satisfies all
these linear equations.

Conversely, let $(b_1,\ldots,b_k)\in R^k$ be a nonzero scalar solution
of the linear subsystem.  Put $N=E_{12}\in M_n(R)$ and $X_i=b_iN$.
Since $N^2=0$, every monomial of degree at least two vanishes on this tuple.
The linear equations hold by the choice of the $b_i$, and at least one
$X_i$ is nonzero.  This proves the equivalence and the stabilization.
If there are no linear equations, any nonzero scalar vector may be used.

For $R=\Q$, the linear criterion is decidable by rational row reduction.
For $R=\Z$, a nonzero rational kernel vector can be multiplied by a common
denominator, so the same criterion applies.  For $R=\N$, the criterion
is equivalent to feasibility over $\Q$ of
\[
 \sum_i a_{ji}b_i=0\ (1\leq j\leq q),\qquad
 b_i\geq0\ (1\leq i\leq k),\qquad \sum_i b_i=1.
\]
Rational linear feasibility is decidable, for example by successive
elimination of variables in linear inequalities.  Clearing denominators
of a feasible vector gives a nonzero vector in $\N^k$, and normalizing
a nonzero vector in $\N^k$ gives a feasible rational vector.  This proves
decidability in all three cases.
\end{proof}

Thus the strict chain of Theorem~\ref{thm:constant-free-order} must not be
asserted for homogeneous problems.  In fact
\[
 \HEq_1(R)\subsetneq\HEq_2(R)=\HEq_3(R)=\cdots,
 \qquad
 \HSys_1(R)\subsetneq\HSys_2(R)=\HSys_3(R)=\cdots:
\]
Zero padding gives the inclusions.  The equation $X^2=0$, with declared
variable list $(X)$, has no nonzero scalar solution over $R$, but has
the solution $E_{12}$ in dimension two.

\subsection{Different decision problems for equations and systems}

The preceding arguments also give a sharp contrast between single
constant-free equations and constant-free systems.  These statements
concern nontrivial solutions throughout.

\begin{proposition}\label{prop:constant-free-decidable}
For every fixed $n\geq2$, the problems $\FEq_n(\Z)$ and $\FEq_n(\Q)$
are decidable.
\end{proposition}

\begin{proof}
Let $P$ be a constant-free polynomial in $k\geq1$ declared variables.
If $k\geq2$, write its linear part as $a_1X_1+\cdots+a_kX_k$.
One linear equation in at least two unknowns has a nonzero rational
kernel vector, hence a nonzero integer kernel vector
$(b_1,\ldots,b_k)$.  Substituting $X_i=b_iE_{12}$ kills the linear
part and all higher-degree monomials.  Thus every such input has a
nontrivial solution over both domains.  This argument also covers variables
declared in the input but absent from $P$.

It remains to decide the one-variable case $p(T)$.  The zero polynomial
has a nonzero solution.  If $p\ne0$ and $T^2\mid p$, use $E_{12}$.
Otherwise $T$ occurs as a factor with multiplicity exactly one, since
$p(0)=0$.  Factor $p$ over $\Q$, writing its other irreducible factors
in monic form.  Over $\Q$, a nonzero matrix zero exists precisely when
one such factor $q\ne T$ has degree at most $n$.  Over $\Z$, the same
criterion holds with the additional requirement $q\in\Z[T]$.

For necessity, a nonzero matrix zero has minimal polynomial dividing $p$.
Its minimal polynomial cannot divide $T$, so it contains an irreducible
factor $q\ne T$ of degree at most $n$.  For an integer matrix, its monic
minimal polynomial and its monic rational irreducible factors have integer
coefficients: the minimal polynomial divides the monic integer
characteristic polynomial, and Gauss's lemma applies.  For sufficiency,
take the companion matrix of the chosen monic factor $q$ and pad it with
zeros to size $n$.  Its entries lie in the required domain, it is nonzero,
and it is annihilated by $p$.  Effective rational factorization, degree
comparison, and coefficient inspection decide the criterion.
\end{proof}

No conclusion about $\FEq_n(\N)$ follows from this proof: a nonzero
integer kernel vector of the linear part need not be nonnegative.

\begin{theorem}\label{thm:constant-free-systems-hard}
For every fixed $n\geq1$ and $R\in\{\N,\Z,\Q\}$, ordinary matrix systems
reduce effectively to constant-free systems with nontrivial solutions:
\[
 \Sys_n(R)\leq_m\FSys_n(R).
\]
Both $\FSys_n(\N)$ and $\FSys_n(\Z)$ are computably enumerable complete,
and, unconditionally,
\begin{equation}\label{eq:rational-constant-free-systems}
 \FSys_n(\Q)\equiv_m\Hten(\Q).
\end{equation}
Consequently, $\FSys_n(\Q)$ is undecidable under
\eqref{eq:openai-undecidability}, and computably enumerable complete under
\eqref{eq:integer-definition}.
\end{theorem}

\begin{proof}
Start with an arbitrary scalar-coefficient matrix system
$P_1(\mathbf X)=0,\ldots,P_q(\mathbf X)=0$ in dimension $n$.
Introduce two matrix variables $J,A$ and impose
\begin{equation}\label{eq:constant-free-identity}
 \begin{gathered}
 J^2=J,\qquad JA=AJ=A,\qquad A^n=2J,\\
 JX_i=X_iJ=X_i\qquad(1\leq i\leq k).
 \end{gathered}
\end{equation}
Each displayed chained equality abbreviates two equations.
For each $P_j$, replace its constant term $cI_n$ by $cJ$, leaving all
positive-degree monomials unchanged; call the resulting polynomial
$P_j^J$.  Append the equations $P_j^J=0$.  Every polynomial in the
new system has zero constant term.

Consider a nontrivial solution of this system.  If $J=0$, the support
equations in \eqref{eq:constant-free-identity} force $A=0$ and every
$X_i=0$, a contradiction.  Thus the idempotent $J$ has rational rank
$\rho$ with $1\leq \rho\leq n$.  The operator $A$ preserves
$\operatorname{im}_{\Q}J$, and its restriction satisfies $A^n=2I$ on
that $\rho$-dimensional rational space.  After choosing a rational basis,
Lemma~\ref{lem:dimension-detector} implies $n\mid \rho$.  Hence $\rho=n$.
A full-rank idempotent is the identity, so $J=I_n$, and the remaining
equations recover the original system.

Conversely, extend any solution of the original system by $J=I_n$ and
$A=T_n$, where $T_n$ is the nonnegative matrix from
Lemma~\ref{lem:dimension-detector}.  The resulting tuple is a nontrivial
solution of the new system over the same domain.  This gives an effective
reduction from $\Sys_n(R)$ to $\FSys_n(R)$ for all three domains.
For $R=\N,\Z$, Corollary~\ref{cor:completeness} proves the lower bound.
All three constant-free system problems are computably enumerable by
enumerating nonzero tuples and checking their equations.

For the rational upper reduction, expand a given constant-free system
entrywise.  Write its scalar residuals as
$p_1(\mathbf x),\ldots,p_\ell(\mathbf x)$, where $\mathbf x$ contains
the entries of every declared matrix variable, including variables absent
from the equations.  Put $s(\mathbf x)=\sum_{a=1}^{kn^2}x_a^2$, and introduce one
new scalar variable $z$.  The single scalar equation
\begin{equation}\label{eq:nontrivial-rational-scalar}
 \sum_{j=1}^{\ell}p_j(\mathbf x)^2+
 \bigl(zs(\mathbf x)-1\bigr)^2=0
\end{equation}
has a rational solution exactly when all residuals vanish and
$s(\mathbf x)\ne0$.  Since rational squares are nonnegative, the latter
condition is equivalent to the original matrix tuple being nontrivial.
This proves $\FSys_n(\Q)\leq_m\Hten(\Q)$.  The reverse reduction follows
from the ordinary-system reduction just proved and
\eqref{eq:intro-rational-equivalence}.  The final assertions follow from
Corollary~\ref{cor:rational-hardness}.
\end{proof}

In particular, for $n\geq2$ there is no computable solvability-preserving
transformation from constant-free systems over $M_n(\Z)$ with nontrivial
solutions to single constant-free equations with nontrivial solutions.
Under \eqref{eq:openai-undecidability}, the same obstruction holds over $M_n(\Q)$:
the source problem is undecidable whereas the target problem is decidable
by Proposition~\ref{prop:constant-free-decidable}.
This does not conflict with Theorem~\ref{thm:systems-one}, which permits
constant terms and concerns ordinary solvability.

\appendix
\section{Explicit choices of the range parameters}
\label{app:parameters}

The quantities $d_n$ and $s_n$ in Proposition~\ref{prop:exact-range}
describe a chosen polynomial, rather than uniquely determined invariants.
This appendix gives explicit choices and their numerical values.  The
scale factors follow from a classical central-polynomial evaluation;
an elementary substitution then reduces the number of variables to three.
The resulting choices are not asserted to be optimal.

\subsection{An explicit multilinear polynomial and its scale factor}
\label{app:regev}

Fix $n\geq2$ and put $q=n^2$.  For two lists of variables
$\mathbf X=(X_1,\ldots,X_q)$ and $\mathbf Y=(Y_1,\ldots,Y_q)$, set
\[
 m_j(\mathbf X)=X_{(j-1)^2+1}\cdots X_{j^2}
 \qquad(1\leq j\leq n),
\]
and define $m_j(\mathbf Y)$ in the same way.  Thus the successive block
lengths are $1,3,\ldots,2n-1$.  For $\sigma\in S_q$, write
$\mathbf X_\sigma=(X_{\sigma(1)},\ldots,X_{\sigma(q)})$.
Define the unnormalized double alternation
\begin{equation}\label{eq:regev-polynomial}
 R_n(\mathbf X,\mathbf Y)=
 \sum_{\sigma,\tau\in S_q}\sgn(\sigma)\sgn(\tau)
 \prod_{j=1}^{n}m_j(\mathbf X_\sigma)m_j(\mathbf Y_\tau).
\end{equation}
This is an integer-coefficient fully multilinear polynomial in $2n^2$
variables, of degree $2n^2$.

For a list $\mathbf A=(A_1,\ldots,A_q)$ of matrices, let
$\Delta_n(\mathbf A)$ denote the determinant of the $q\times q$ matrix
whose columns are the coordinates of the $A_i$ in the ordered basis
\[
 E_{11},E_{12},\ldots,E_{1n},E_{21},\ldots,E_{nn}.
\]
The Regev--Formanek evaluation gives
\begin{equation}\label{eq:regev-evaluation}
 R_n(\mathbf A,\mathbf B)
 =(-1)^{n-1}\delta_n\,
 \Delta_n(\mathbf A)\Delta_n(\mathbf B)I_n,
\end{equation}
where
\begin{equation}\label{eq:delta-formula}
 \boxed{\displaystyle
 \delta_n=\frac{1}{2n-1}
 \left(\prod_{j=1}^{n}\frac{(2j-1)!}{j!}\right)^2.}
\end{equation}
See Formanek~\cite{Formanek1987} and the explicit normalization in
Procesi~\cite[Remark~4.5 and Theorem~4.13]{Procesi2022}; the polynomial
and its evaluation are equations~(35)--(37) in the cited author version.
The determinants in \eqref{eq:regev-evaluation} describe the values of
\eqref{eq:regev-polynomial}; determinant is not added to the language of
the equations.

For $n\geq2$ the positive integer in \eqref{eq:delta-formula} can also
be written as
\begin{equation}\label{eq:delta-integrality}
 \delta_n=(2n-1)\left(
 \frac{(2n-2)!}{n!}
 \prod_{j=1}^{n-1}\frac{(2j-1)!}{j!}\right)^2.
\end{equation}
Every factorial quotient inside the parentheses is an integer.
Equation~\eqref{eq:regev-evaluation} shows that all integral values of
$(-1)^{n-1}R_n$ lie in $\delta_n\Z I_n$.  Conversely, put the ordered
elementary matrix basis in both lists, and multiply the first matrix of one
list by any $t\in\N$.  The value is then $\delta_ntI_n$.
Interchanging two entries of the other list changes its sign.  All these
arguments are nonnegative matrices, so
\begin{equation}\label{eq:regev-range}
 \begin{aligned}
 (-1)^{n-1}R_n\bigl(M_n(\N)^{2n^2}\bigr)
 &=(-1)^{n-1}R_n\bigl(M_n(\Z)^{2n^2}\bigr)\\
 &=\delta_n\Z I_n.
 \end{aligned}
\end{equation}
Thus $(\delta_n,2n^2)$ is already an explicit admissible pair.
On tuples of elementary matrices, the values of $R_n$ are exactly
$0,\delta_nI_n,-\delta_nI_n$.  In particular, its basis-value gcd in
the proof of Proposition~\ref{prop:exact-range} is $\delta_n$; the extra
B\'ezout blocks are unnecessary for this particular polynomial.

\subsection{Compression to three variables}
\label{app:compression}

The preceding polynomial can be replaced by one in only three matrix
variables without changing its scale factor.

\begin{proposition}\label{prop:three-variables}
For every $n\geq2$ there is an explicitly defined homogeneous polynomial
$\widetilde Q_n\in\free{U,S,T}$ such that
\begin{equation}\label{eq:compressed-range}
 \widetilde Q_n\bigl(M_n(\N)^3\bigr)
 =\widetilde Q_n\bigl(M_n(\Z)^3\bigr)
 =\delta_n\Z I_n.
\end{equation}
Its degree is $2n^3+n+1$.
\end{proposition}

\begin{proof}
For $1\leq a,b\leq n$, define polynomial expressions
\begin{align}
 W_{ab}(S,T)&=S^{n-a}TS^{b-1},\label{eq:W-definition}\\
 K_{ab}(U,S,T)&=W_{a1}(S,T)\,U\,W_{1b}(S,T)
                 -W_{a2}(S,T)\,U\,W_{2b}(S,T).
 \label{eq:K-definition}
\end{align}
Let $\mathbf W$ be the list of all $W_{ab}$ in the same order as the
elementary matrix basis above, and let $\mathbf W^*$ be that list with its
first entry $W_{11}$ replaced by $K_{11}$.  Set
\begin{equation}\label{eq:compressed-polynomial}
 \widetilde Q_n(U,S,T)=(-1)^{n-1}R_n(\mathbf W^*,\mathbf W).
\end{equation}
After expansion this is a polynomial in the three variables $U,S,T$
with integer coefficients.  For integral values of these variables,
every expression substituted into $R_n$ is an integral matrix.
Consequently \eqref{eq:regev-evaluation} puts every value of
$\widetilde Q_n$ in $\delta_n\Z I_n$.

For the reverse inclusion, use the fixed nonnegative matrices
\begin{equation}\label{eq:compression-witnesses}
 S_0=\sum_{i=1}^{n-1}E_{i,i+1},\qquad T_0=E_{n1}.
\end{equation}
The elementary matrix multiplication rules give
\[
 W_{ab}(S_0,T_0)=E_{ab},\qquad
 K_{ab}(U,S_0,T_0)=(U_{11}-U_{22})E_{ab}.
\]
Here $U_{11}$ and $U_{22}$ denote entries of the assigned matrix $U$.
Thus $\Delta_n(\mathbf W)=1$ and
$\Delta_n(\mathbf W^*)=U_{11}-U_{22}$ at these assignments, and
\begin{equation}\label{eq:compressed-witness-value}
 \widetilde Q_n(U,S_0,T_0)=\delta_n(U_{11}-U_{22})I_n.
\end{equation}
For any $z\in\Z$, write $z=z_+-z_-$ with $z_+,z_-\in\N$ and take
$U=\operatorname{diag}(z_+,z_-,0,\ldots,0)$.  This realizes
$\delta_nzI_n$ using only nonnegative matrix arguments, proving
\eqref{eq:compressed-range}.

Finally, the expressions in \eqref{eq:W-definition} are homogeneous, with
\[
 \deg W_{ab}=n-a+b,\qquad
 \sum_{a,b=1}^n\deg W_{ab}=n^3,
 \qquad \deg K_{11}=2n+1.
\]
Every monomial of $R_n$ uses every slot once.  Replacing one $W_{11}$,
of degree $n$, by $K_{11}$ therefore makes
\eqref{eq:compressed-polynomial} homogeneous of degree
$2n^3-n+(2n+1)=2n^3+n+1$.  Its nonzero evaluation in
\eqref{eq:compressed-witness-value} shows that it is not the zero
polynomial, so this is its exact degree.
\end{proof}

The elementary matrices in \eqref{eq:compression-witnesses} are witnesses,
not parameters in \eqref{eq:compressed-polynomial}.  Also, $K_{11}$ may
have negative entries even when $U,S,T$ are nonnegative.  It is a signed
polynomial expression, not a quantified matrix variable.  The containment
argument uses the range of $R_n$ on integral arguments, in accordance
with the ambient-ring convention of Definition~\ref{def:equation}.

\subsection{Values in small dimensions and general growth}
\label{app:table}

We may now make one explicit choice for each dimension:
\begin{equation}\label{eq:preferred-choices}
 Q_1(U,V)=U-V,\qquad Q_2(X,Y)=B_2(X,Y),\qquad
 Q_n=\widetilde Q_n\quad(n\geq3).
\end{equation}
Propositions~\ref{prop:two-variables} and \ref{prop:three-variables} give
\begin{equation}\label{eq:preferred-parameters}
 (d_1,s_1)=(1,2),\qquad (d_2,s_2)=(1,2),\qquad
 (d_n,s_n)=(\delta_n,3)\quad(n\geq3).
\end{equation}
The first few values, together with the degrees of these polynomials, are
shown in Table~\ref{tab:parameters}.

\begin{table}[htbp]
\centering
\renewcommand{\arraystretch}{1.15}
\begin{tabular}{crcc}
\toprule
$n$ & $d_n$ & $s_n$ & $\deg Q_n$ \\
\midrule
1 & $1$ & $2$ & $1$ \\
2 & $1$ & $2$ & $4$ \\
3 & $720$ & $3$ & $58$ \\
4 & $22\,680\,000$ & $3$ & $133$ \\
5 & $161\,310\,320\,640\,000$ & $3$ & $256$ \\
\bottomrule
\end{tabular}
\caption{Explicit admissible choices for the exact-range proposition.
The entries describe the polynomials in \eqref{eq:preferred-choices};
no minimality is asserted.}
\label{tab:parameters}
\end{table}

For the uniform scale factors $\delta_n$ of
\eqref{eq:delta-formula}, cancellation of consecutive products gives
\begin{equation}\label{eq:delta-recurrence}
 \frac{\delta_{n+1}}{\delta_n}
 =\frac{2n-1}{2n+1}
 \left(\frac{(2n+1)!}{(n+1)!}\right)^2
 \qquad(n\geq2).
\end{equation}
Here $\delta_2=3$, whereas the preferred dimension-two polynomial has
$d_2=1$.  Thus \eqref{eq:delta-recurrence} describes the preferred
$d_n$ only from $n=3$ onwards.

Here $\log$ and $\exp$ denote the natural logarithm and exponential.
For sequences $u_n$ and eventually positive $v_n$, the notation
$u_n=O(v_n)$ means $|u_n|\leq Cv_n$ for some constant $C>0$ and all
sufficiently large $n$.  For eventually positive sequences,
$u_n=\Theta(v_n)$ means $cv_n\leq u_n\leq Cv_n$ eventually, for
some constants $0<c\leq C$.  As $n\to\infty$, Stirling's estimate yields
\begin{equation}\label{eq:delta-asymptotic}
 \log\delta_n
 =n^2\log n+\left(2\log2-\frac32\right)n^2+O(n\log n).
\end{equation}
Indeed,
\[
 \log\delta_n
 =2\sum_{j=1}^n\bigl(\log((2j-1)!)-\log(j!)\bigr)-\log(2n-1),
\]
and the summand is
$j\log j+(2\log2-1)j+O(\log(j+1))$.
Using
$\sum_{j=1}^n j\log j=\tfrac12 n^2\log n-\tfrac14n^2+O(n\log n)$
proves \eqref{eq:delta-asymptotic}.
Thus this construction has a constant number of variables, degree cubic
in $n$, and scale factor of order $\exp(\Theta(n^2\log n))$.
This growth describes these explicit choices; it is not a lower bound
for all polynomials satisfying \eqref{eq:exact-range}.

\begin{remark}[Sharing the fixed witnesses]\label{rem:shared-witnesses}
The main reduction uses copies on pairwise disjoint variable sets so
that their values can be prescribed independently.  The explicit family
above also allows a smaller total number of variables.  For $n\geq3$,
substitute
\[
 \widetilde Q_n(U_1,S,T),\ldots,\widetilde Q_n(U_k,S,T)
\]
into \eqref{eq:main-polynomial}, with $d=\delta_n$.  These expressions
intentionally share the variables $S,T$; they are not copies on disjoint
variable sets.  Every evaluation is still in $\delta_n\Z I_n$, and
\eqref{eq:compressed-witness-value} realizes any desired integer tuple
by the common choices $S=S_0,T=T_0$ and independent choices of the $U_i$.
The proof of Theorem~\ref{thm:main-reduction} therefore applies with
$k+2$ matrix variables.
For $n=2$, the same argument uses
$B_2(X,Y_1),\ldots,B_2(X,Y_k)$ and the fixed matrix $X_0$ from
Proposition~\ref{prop:two-variables}, giving $k+1$ matrix variables.
\end{remark}

\section*{Acknowledgments}

This paper was produced in an AI-assisted workflow. AI assistants helped
explore candidate constructions, draft proofs, and write software for
numerical experiments.


\begin{thebibliography}{99}

\bibitem{BresarDrensky2013}
M.~Bre\v{s}ar and V.~Drensky,
Central polynomials for matrices over finite fields,
\emph{Linear and Multilinear Algebra} \textbf{61} (2013), no.~7, 939--944.
\href{https://doi.org/10.1080/03081087.2012.716434}{doi:10.1080/03081087.2012.716434}.
Preprint: \href{https://arxiv.org/abs/1205.5053}{arXiv:1205.5053}.

\bibitem{DPR1961}
M.~Davis, H.~Putnam, and J.~Robinson,
The decision problem for exponential Diophantine equations,
\emph{Annals of Mathematics} (2) \textbf{74} (1961), no.~3, 425--436.
\href{https://doi.org/10.2307/1970289}{doi:10.2307/1970289}.

\bibitem{Formanek1972}
E.~Formanek,
Central polynomials for matrix rings,
\emph{Journal of Algebra} \textbf{23} (1972), no.~1, 129--132.
\href{https://doi.org/10.1016/0021-8693(72)90050-6}{doi:10.1016/0021-8693(72)90050-6}.

\bibitem{Formanek1987}
E.~Formanek,
A conjecture of Regev about the Capelli polynomial,
\emph{Journal of Algebra} \textbf{109} (1987), no.~1, 93--114.
\href{https://doi.org/10.1016/0021-8693(87)90166-9}{doi:10.1016/0021-8693(87)90166-9}.

\bibitem{GlazerProject2026Q}
\emph{$\Z$ is Diophantine in $\Q$: Version 16},
research manuscript, \texttt{diophantine-z-q} project, November 11, 2026.
\href{https://github.com/elliotglazer/diophantine-z-q/blob/95cb356d7cd1a5d56d10abeae60497d54d91197e/manuscripts/quartic_v16.pdf}{Version 16 manuscript}.
The released manuscript has no author byline; the project credits collective
contributions and is coordinated by Elliot Glazer.

\bibitem{Hilbert1909}
D.~Hilbert,
Beweis f\"ur die Darstellbarkeit der ganzen Zahlen durch eine feste Anzahl
$n$-ter Potenzen (Waringsches Problem),
\emph{Mathematische Annalen} \textbf{67} (1909), no.~3, 281--300.
\href{https://doi.org/10.1007/BF01450405}{doi:10.1007/BF01450405}.

\bibitem{KlepTressl2025}
I.~Klep and M.~Tressl,
A model theoretic perspective on matrix rings,
\emph{Mathematische Zeitschrift} \textbf{309} (2025), article~45.
\href{https://doi.org/10.1007/s00209-024-03671-w}{doi:10.1007/s00209-024-03671-w}.

\bibitem{Lang2002}
S.~Lang,
\emph{Algebra}, revised third edition,
Graduate Texts in Mathematics, vol.~211,
Springer, New York, 2002.
\href{https://doi.org/10.1007/978-1-4613-0041-0}{doi:10.1007/978-1-4613-0041-0}.

\bibitem{Lee2024}
W.~Lee,
The Diophantine problem and bi-interpretability in the classic matrix rings,
conference abstract, \emph{Groups, Logic, and Computation},
Stevens Institute of Technology, June 2024.
Based on joint work with A.~Miasnikov.
\url{https://web.stevens.edu/algebraic/Stevens2024/Abstracts/Abstract_Lee.pdf}.

\bibitem{Matiyasevich1971}
Yu.~V.~Matiyasevich,
Diophantine representation of enumerable predicates,
\emph{Mathematics of the USSR-Izvestiya} \textbf{5} (1971), no.~1, 1--28.
\href{https://doi.org/10.1070/IM1971v005n01ABEH001004}{doi:10.1070/IM1971v005n01ABEH001004}.

\bibitem{Matiyasevich1993}
Yu.~V.~Matiyasevich,
\emph{Hilbert's Tenth Problem},
Foundations of Computing, MIT Press, Cambridge, MA, 1993.
\href{https://mitpress.mit.edu/9780262132954/hilberts-10th-problem/}{ISBN 0-262-13295-8}.

\bibitem{Nathanson1987}
M.~B.~Nathanson,
A short proof of Cauchy's polygonal number theorem,
\emph{Proceedings of the American Mathematical Society} \textbf{99}
(1987), no.~1, 22--24.
\href{https://doi.org/10.1090/S0002-9939-1987-0866422-3}{doi:10.1090/S0002-9939-1987-0866422-3}.

\bibitem{OpenAI2026Q}
OpenAI,
\emph{Hilbert's tenth problem over the rational numbers},
OpenAI Math Release preprint, October 6, 2026.
\href{https://github.com/openai/math/blob/main/preprints/Hilberts-tenth-problem-over-the-rational-numbers-October-6-2026/main.pdf}{Author's PDF}.

\bibitem{Procesi2022}
C.~Procesi,
A construction of swap or switch polynomials,
\emph{Advances in Mathematics} \textbf{400} (2022), article~108308.
\href{https://doi.org/10.1016/j.aim.2022.108308}{doi:10.1016/j.aim.2022.108308}.
Author version, with the numbering cited here:
\href{https://arxiv.org/abs/2102.10657v2}{arXiv:2102.10657v2}.

\bibitem{Prunescu2006}
M.~Prunescu,
Undecidable and decidable restrictions of Hilbert's Tenth Problem:
images of polynomials vs.\ images of exponential functions,
\emph{Mathematical Logic Quarterly} \textbf{52} (2006), no.~1, 14--19.
\href{https://doi.org/10.1002/malq.200510013}{doi:10.1002/malq.200510013}.

\bibitem{Razmyslov1973}
Yu.~P.~Razmyslov,
On a problem of Kaplanskii,
\emph{Mathematics of the USSR-Izvestiya} \textbf{7} (1973), no.~3, 479--496.
\href{https://doi.org/10.1070/IM1973v007n03ABEH001952}{doi:10.1070/IM1973v007n03ABEH001952}.

\bibitem{Schnirelmann1933}
L.~Schnirelmann,
\"Uber additive Eigenschaften von Zahlen,
\emph{Mathematische Annalen} \textbf{107} (1933), 649--690.
\href{https://doi.org/10.1007/BF01448914}{doi:10.1007/BF01448914}.

\bibitem{Sun2026Q}
Z.-W.~Sun,
\emph{86 unknowns for Hilbert's tenth problem over the rationals},
preprint, \href{https://arxiv.org/abs/2610.20691v2}{arXiv:2610.20691v2}
[math.NT], October 19, 2026.

\end{thebibliography}
\end{document}